\documentclass{article}

\usepackage[english]{babel}
\usepackage{amssymb,amsfonts,amsmath,amsthm}
\usepackage{graphicx,color}
\usepackage{epsfig}

\newtheorem{theorem}{Theorem}[section]

\newtheorem{proposition}{Proposition}[section]

\theoremstyle{remark}
\newtheorem{remark}{Remark}[section]

\numberwithin{equation}{section}

\newcommand\sap{\mathrel{\Bumpeq\!\!\!\!\!\!\!\longrightarrow}}
\newcommand\finis{\hfill$\triangleleft$}

\begin{document}

\title{Multiple periodic solutions for one-sided sublinear systems: A refinement of the Poincar\'{e}-Birkhoff approach}

\author{{\Large Tobia Dond\`{e} and Fabio Zanolin}
\\
{}
\\
\vspace{2mm} {\it\small Department of Mathematics, Computer
Science and Physics}
\\
{\it\small University of Udine}, {\it\small  Via delle Scienze
206, 33100 Udine  - Italy} }
\date{}

\maketitle

\begin{abstract}
In this paper we prove the existence of multiple periodic solutions (harmonic and subharmonic)
for a class of planar Hamiltonian systems which include the case of the second order scalar ODE $x'' + a(t)g(x) = 0$
with $g$ satisfying a one-sided condition of sublinear type.  We consider the classical approach based on the Poincar\'{e}-Birkhoff
fixed point theorem as well as some refinements on the side of the theory of bend-twist maps and topological horseshoes.
The case of complex dynamics is investigated, too.


\end{abstract}

\section{Introduction}\label{section-1}

The Poincar\'{e}-Birkhoff fixed point theorem deals with a planar homeomorphism $\Psi$
defined on an annular region $A$, such that $\Psi$ is area-preserving, leaves the boundary of $A$
invariant and rotates the two components of $\partial A$ in opposite directions
(twist condition). Under these assumptions, in 1912 Poincar\'{e} conjectured (and proved in some particular cases)
the existence of at least two fixed points for $\Psi$, a result known as ``the Poincar\'{e} last geometric theorem''.
A proof for the existence of at least one fixed point (and actually two in a non-degenerate situation) was obtained by Birkhoff in 1913
\cite{Bi-13}.
In the subsequent years Birkhoff reconsidered the theorem as well as its possible extensions to a more general setting,
for instance, removing the assumption of boundary invariance, or proposing some hypotheses of topological nature
instead of the area-preserving condition, thus opening a line of research that is still active today (see for example \cite{Ca-82}, \cite{Bon-12},
and references therein). The skepticism of some mathematicians about the correctness of the proof of the second fixed point
motivated Brown and Neumann  to present in \cite{BN-77} a full detailed proof, adapted from Birkhoff's 1913 paper,
in order to eliminate previous possible controversial aspects. Another approach for the proof of the second fixed point
has been proposed in \cite{Sl-93}, coupling \cite{Bi-13} with a result for removing fixed points of zero index.

In order to express the twist condition in a more precise manner, the statement of the Poincar\'{e}-Birkhoff
theorem is usually presented in terms of the lifted map $\tilde{\Psi}$. Let us first introduce some notation.
Let $D(R)$ and $D[R]$ be, respectively, the open and the closed disc
of center the origin and radius $R > 0$ in ${\mathbb R}^2$ endowed with the Euclidean norm $||\cdot||.$
Let also $C_R:= \partial D(R).$
Given $0 < r < R,$ we denote by $A$ or $A[r,R]$ the closed annulus $A[r,R]:= D[R]\setminus D(r).$ Hence the area-preserving
(and orientation-preserving) homeomorphism
$\Psi: A\to \Psi(A)=A$ is lifted to a map
$\tilde{\Psi}: \tilde{A}\to \tilde{A},$ where $\tilde{A}:= {\mathbb R}\times[r,R]$
is the covering space of $A$ via the covering projection $\Pi: (\theta,\rho)\mapsto (\rho\cos\theta,\rho\sin\theta)$ and
\begin{equation}\label{eq-0.1}
\tilde{\Psi}: (\theta,\rho)\mapsto (\theta + 2\pi \mathcal{J}(\theta,\rho), \mathcal{S}(\theta,\rho)),
\end{equation}
with the functions $\mathcal{J}$ and $\mathcal{S}$ being $2\pi$-periodic in the $\theta$-variable.
Then, the classical (1912-1913) Poincar\'{e}-Birkhoff fixed point theorem can be stated as follows (see \cite{BN-77}).
\begin{theorem}\label{pb-01}
Let $\Psi: A\to \Psi(A) = A$ be an area preserving homeomorphism such that the following two conditions are satisfied:
\begin{description}
\item[$(PB1)\;$] $\mathcal{S}(\theta,r)= r, \quad \mathcal{S}(\theta,R)= R, \quad \forall\, \theta\in {\mathbb R};$
\item[$(PB2)\;$] $\exists \, j\in {\mathbb Z}: \; (\mathcal{J}(\theta,r) - j)(\mathcal{J}(\theta,R) - j) < 0, \quad \forall\, \theta\in {\mathbb R}.$
\end{description}
Then $\Psi$ has at least two fixed points $z_1, z_2$ in the interior of $A$ and $\mathcal{J}(\theta,\rho)=j$ for $\Pi(\theta,\rho) = z_i\,.$
\end{theorem}
We refer to condition $(PB1)$ as to the ``boundary invariance'' and we call $(PB2)$ the ``twist condition''.
The function $\mathcal{J}$ can be regarded as a rotation number associated with the points.
In the original formulation of the theorem it is $j=0,$ however any integer $j$ can be considered.

The Poincar\'{e}-Birkhoff theorem is a fundamental result in the areas of fixed point theory and dynamical systems,
as well as in their applications to differential equations. General presentations can be found in \cite{MoZh-05}, \cite{MeHa-92} and \cite{LC-11}.
There is a large literature on the subject and certain subtle and delicate
points related to some controversial extensions of the theorem have been settled only in recent years (see \cite{Re-97},\cite{MU-07},\cite{LCW-10}).
In the applications to the study of periodic non-autonomous planar Hamiltonian systems, the map $\Psi$ is often the Poincar\'{e} map
(or one of its iterates). In this situation the condition of boundary invariance is usually not satisfied, or very difficult to prove: as a consequence, variants of the Poincar\'{e}-Birkhoff theorem in which the hypothesis $(PB1)$ is not required turn out to be quite useful
for the applications (see \cite{DaRe-02} for a general discussion on this topic).
As a step in this direction we present the next result, following from W.Y. Ding in \cite{Di-82}.
\begin{theorem}\label{pb-02}
Let $\Psi: D[R]\to \Psi(D[R])\subseteq {\mathbb R}^2$ be an area preserving homeomorphism with $\Psi(0) = 0$ and such that the twist condition
$(PB2)$ holds.
Then $\Psi$ has at least two fixed points $z_1, z_2$ in the interior of $A$ and $\mathcal{J}(\theta,\rho)=j$ for $\Pi(\theta,\rho) = z_i\,.$
\end{theorem}
\noindent
The proof in \cite{Di-82} (see also \cite[Appendix]{DiZa-92}) relies on the Jacobowitz version of the Poincar\'{e}-Birkhoff theorem
for a pointed topological disk \cite{Ja-76}, \cite{Ja-77} which was corrected in \cite{LCW-10}, since the result is true for strictly star-shaped
pointed disks and not valid in general, as shown by a counterexample in the same article. Another (independent) proof of Theorem \ref{pb-02} was
obtained by Rebelo in \cite{Re-97}, who brought the proof back to that of Theorem \ref{pb-01} and thus to the ``safe'' version of Brown and Neumann
\cite{BN-77}. Other versions of the Poincar\'{e}-Birkhoff theorem
giving Theorem \ref{pb-02} as a corollary can be found in \cite{Fr-88},\cite{Fr-06},\cite{QT-05},\cite{Ma-13}
(see also \cite[Introduction]{FSZ-12} for a general discussion about
these delicate aspects).
For Poincar\'{e} maps associated with Hamiltonian systems there is a much more general version of the theorem due to Fonda and Ure\~{n}a
in \cite{FU-16},\cite{FU-17}, which will be recalled later in the paper with some more details.

In \cite{Di-07},\cite{Di-12}, T.R. Ding proposed a variant of the Poincar\'{e}-Birkhoff theorem, by introducing the concept of ``bend-twist map''. Given
a continuous map $\Psi: A \to \Psi(A)\subseteq {\mathbb R}^2\setminus\{0\},$ which admits a lifting $\tilde{\Psi}$ as in \eqref{eq-0.1}, we define
$$\Upsilon(\theta,\rho):= \mathcal{S}(\theta,\rho) - \rho.$$
We call $\Psi$ a \textit{bend-twist map} if it $\Psi$ satisfies the twist condition
and $\Upsilon$ changes its sign on a non-contractible Jordan closed curve $\Gamma$ contained in
the set of points in the interior of $A$ where $\mathcal{J} = j.$ The original treatment was given in \cite{Di-07} for analytic maps. There are extensions
to continuous maps as well \cite{PaZa-11}, \cite{PaZa-13}. Clearly, the bend-twist map condition is difficult to check in practice, due to the lack
of information about the curve $\Gamma$ (which, in the non-analytic case, may not even be a curve). For this reason,
one can rely on the following corollary \cite[Corollary 7.3]{Di-07} which also follows from the Poincar\'{e}-Miranda theorem
(as observed in \cite{PaZa-11}).

\begin{theorem}\label{pb-03}
Let $\Psi: A= A[r,R]\to \Psi(A)\subseteq {\mathbb R}^2\setminus\{0\}$ be a continuous map such that the twist condition
$(PB2)$ holds. Suppose that there are two disjoint arcs $\alpha, \beta$ contained in $A,$ connecting the inner with the outer boundary of the annulus and such that
\begin{description}
\item[$(BT1)\;$] $\Upsilon > 0$ on $\alpha$ and $\Upsilon < 0$ on  $\beta.$
\end{description}
Then $\Psi$ has at least two fixed points $z_1, z_2$ in the interior of $A$ and $\mathcal{J}(\theta,\rho)=j$ for $\Pi(\theta,\rho) = z_i\,.$
\end{theorem}

A simple variant of the above theorem considers $2n$ pairwise disjoint simple arcs $\alpha_i$ and $\beta_i$ (for $i=1,\dots,n$) contained in $A$
and connecting the inner with the outer boundary. We label these arcs in cyclic order so that each $\beta_i$ is between $\alpha_i$ and
$\alpha_{i+1}$ and each $\alpha_i$ is between $\beta_{i-1}$ and $\beta_{i}$ (with $\alpha_{n+1} =\alpha_{1}$ and $\beta_{0} = \beta_{n}$)
and suppose that
\begin{description}
\item[$(BTn)\;$] $\Upsilon > 0$ on $\alpha_i$ and $\Upsilon < 0$ on $\beta_i\,,$ for all $i=1,\dots,n.$
\end{description}
Then $\Psi$ has at least $2n$ fixed points $z_i$ in the interior of $A$ and $\mathcal{J}(\theta,\rho)=j$ for $\Pi(\theta,\rho) = z_i\,.$
These results also apply in the case of a topological annulus (namely, a compact planar set homeomorphic to $A$)
and do not require that $\Psi$ is area-preserving and also the assumption of $\Psi$ being a homeomorphism is not required, as continuity is enough.
Moreover, since the fixed points are obtained in regions with
index $\pm 1,$ the results are robust with respect to small (continuous) perturbations of the map $\Psi.$

A special case in which condition $(BT1)$ holds is when $\Psi(\alpha) \in D(r)$ and $\Psi(\beta) \in {\mathbb R}^2 \setminus D[R],$
namely, the annulus $A,$ under the action of the map $\Psi,$ is not only twisted, but also strongly stretched, in the sense that
there is a portion of the annulus around the curve $\alpha$ which is pulled inward near the origin inside the disc $D(r)$,
while there is a portion of the annulus around the curve $\beta$ which is pushed outside the disc $D[R].$ This special situation
where a strong bend and twist occur is reminiscent of the geometry of the \textit{Smale horseshoe maps} \cite{Sm-67},\cite{Mo-73} and, indeed, we
will show how to enter in a variant of the theory of \textit{topological horseshoes} in the sense of Kennedy and Yorke \cite{KY-01}.
To this aim, we recall a few definitions which are useful for the present setting.
By a \textit{topological rectangle} we mean a subset ${\mathcal R}$ of the plane which is
homeomorphic to the unit square. Given an arbitrary topological rectangle ${\mathcal R}$ we can define an orientation, by selecting two
disjoint compact arcs on its boundary. The union of these arcs is denoted by ${\mathcal R}^-$ and the pair
${\widehat{\mathcal R}}:= ({\mathcal R},{\mathcal R}^-)$ is called an \textit{oriented rectangle}. Usually the two components of
${\mathcal R}^-$ are labelled as the left and the right sides of ${\widehat{\mathcal R}}$. Given two oriented rectangles $\widehat{{\mathcal A}}$,
$\widehat{{\mathcal B}}$, a continuous map $\Psi$ and a compact set $H\subseteq \text{dom}(\Psi)\cap {\mathcal A},$ the notation
$$(H,\Psi): \widehat{{\mathcal A}}\sap \widehat{{\mathcal B}}$$
means that the following ``stretching along the paths'' (SAP)
property is satisfied: \textit{any path $\gamma,$ contained
in ${\mathcal A}$ and joining the opposite sides of ${\mathcal A}^-,$ contains a sub-path $\sigma$ in $H$ such that the image of $\sigma$
through $\Psi$ is a path contained in ${\mathcal B}$ which connects the opposite sides of ${\mathcal B}^-$.}
We also write $\Psi: \widehat{{\mathcal A}}\sap \widehat{{\mathcal B}}$ when $H = {\mathcal A}.$
By a path $\gamma$ we mean a continuous map defined on a compact interval. When, loosely speaking, we say that a path is contained in a given set
we actually refer to its image $\bar{\gamma}.$ Sometimes it will be useful to consider a relation of the form
$$\Psi:\widehat{{\mathcal A}}\sap^{k} \widehat{{\mathcal B}},$$
for $k\geq 2$ a positive integer, which means that there are at least $k$ compact subsets $H_{1}\,,\dots,H_{k}$ of ${\mathcal A}$
such that $(H_{i},\Psi): \widehat{{\mathcal A}}\sap \widehat{{\mathcal B}}$ for all $i=1,\dots,k.$ From the results in
\cite{PaZa-04b}, \cite{PaZa-04a} we have that $\Psi$ has a fixed point in $H$ whenever
$(H,\Psi): \widehat{{\mathcal R}}\sap \widehat{{\mathcal R}}$. If for a rectangle ${\mathcal R}$ we have that
$\Psi:\widehat{{\mathcal A}}\sap^{k} \widehat{{\mathcal B}},$ for $k\geq 2$, then $\Psi$ has at least $k$ fixed points in ${\mathcal R}.$
In this latter situation, one can also prove the presence of chaotic-like dynamics of coin-tossing type
(this will be briefly discussed later).

The aim of this paper is to analyze, under these premises, a simple example of planar system with periodic coefficients of the form
$$
\begin{cases}
x'=h(y)\\
y'=-q(t)g(x),
\end{cases}
$$
including the second order scalar equation of Duffing type
$$x'' + q(t) g(x) = 0.$$
The prototypical nonlinearity we consider is a function which changes sign at zero and is bounded only on one-side, such as
$g(x) = -1 + \exp(x).$
We do not assume that the weight function $q(t)$ is of constant sign, but, for simplicity,
we suppose that $q(\cdot)$ has a positive hump followed by a negative one.
We prove the presence of periodic solutions coming in pairs
(Theorem \ref{th-A} in Section \ref{section-2}, following the Poincar\'{e}-Birkhoff theorem) or coming in quadruplets
(Theorem \ref{th-B} in Section \ref{section-2}, following bend-twist maps and SAP techniques), the latter
depending on the intensity of the negative part of $q(\cdot).$ To this purpose, we shall express the weight function as
\begin{displaymath}
q(t)= a_{\lambda,\mu}(t):=\lambda a^+(t)-\mu a^-(t),\qquad\lambda, \mu > 0,
\end{displaymath}
being $a(\cdot)$ a periodic sign changing function.

The plan of the paper is the following. In Section \ref{section-2} we present our main results (Theorem \ref{th-A} and Theorem \ref{th-B})
for the existence and multiplicity of periodic solutions. In Section \ref{section-3}, we provide simplified proofs in the special case of a stepwise weight function:
this allows us to highlight the geometric structure underlying the theorems.
These proofs can considered as preparatory to the general ones given in Section \ref{section-4}. Another advantage of considering this
particular framework lies on the fact that a stepwise weight produces a switched system made by two autonomous equations and therefore, in this case,
some threshold constants for $\lambda$ and $\mu$ can be explicitly computed. In Section \ref{section-5}
we show how to extend our main results to the case of subharmonic solutions. Eventually, Section \ref{section-6} concludes the paper with
a list of some possible applications.

\section{Statement of the main results}\label{section-2}

This work deals with the existence and multiplicity of periodic
solutions to sign-indefinite nonlinear first order planar systems
of the form
\begin{equation}\label{eq-1.1}
\begin{cases}
x'=h(y)\\
y'=-a_{\lambda,\mu}(t)g(x).
\end{cases}
\end{equation}
Throughout the article, we will suppose that $h,g: {\mathbb R}
\to {\mathbb R}$ are locally Lipschitz continuous functions
satisfying the following assumptions:
\begin{equation*}
\begin{array}{lll}
& h(0)=0, \; h(y)y>0\textrm{ for all }y\neq0 \\
& g(0)=0, \; g(x)x>0\textrm{ for all }x\neq0\\
& \displaystyle{h_{0}:= \liminf_{|y|\to 0}\frac{h(y)}{y}>0,\quad g_{0}:= \liminf_{|x|\to 0}\frac{g(x)}{x} > 0.}
\end{array}
\leqno{(C_{0})}
\end{equation*}

We will also suppose that \emph{at least one} of the following conditions holds
$$
h \; \text{is bounded on }\; {\mathbb R}^{-},
\leqno{(h_{-})}
$$
$$
h \; \text{is bounded on }\; {\mathbb R}^{+},
\leqno{(h_{+})}
$$
$$
g \; \text{is bounded on }\; {\mathbb R}^{-},
\leqno{(g_{-})}
$$
$$
g \; \text{is bounded on }\; {\mathbb R}^{+}.
\leqno{(g_{+})}
$$
We also set
$$\mathcal{G}(x):= \int_{0}^{x} g(\xi)d\xi,\quad \mathcal{H}(y):= \int_{0}^{y} h(\xi)d\xi.$$
Concerning the weight function $a_{\lambda,\mu}(t)\,,$ it is defined starting from a $T$-periodic sign-changing map
$a: {\mathbb R} \to {\mathbb R}$ by setting
\begin{displaymath}
a_{\lambda,\mu}(t)=\lambda a^+(t)-\mu a^-(t),\qquad\lambda, \mu > 0.
\end{displaymath}
As usual, $a^+:= (a + |a|)/2$ is the positive part of $a(\cdot)$ and $a^-:= a^+ - a.$ Given an interval $I,$ by $a \succ 0$ on $I$
we mean that $a(t)\geq 0$ for almost every $t\in I$ with $a>0$ on a subset of $I$ of positive measure. Similarly, $a \prec 0$ on $I$ means that $-a \succ 0$ on $I.$
For sake of simplicity, we suppose that in a period the weight function $a(t)$ has one positive hump followed by one negative hump.
The case in which $a(t)$ has several (but finite) changes of sign in a period could be dealt with as well, but will be not treated here.
Therefore, we suppose that there are $t_{0}$ and $T_1\in \,]0,T[$
such that
$$
a \succ 0 \quad \text{ on } \; [t_{0},t_{0}+T_1] \quad \text{and } \quad a \prec 0 \quad \text{ on } \; [t_{0}+T_1,t_{0}+T].
\leqno{(a^*)}$$
Actually, due to the $T$-periodicity of the weight function, it will be not restrictive to take
$t_{0} = 0$ and we will assume it for the rest of the paper.
Concerning the regularity of the weight function, we suppose that $a(\cdot)$ is continuous (or piecewise-continuous), although from the proofs
it will be clear that all the results
are still valid for $a^+\in L^{\infty}([0,T_1])$ and $a^-\in L^1([T_1,T]).$

The assumptions on $h$ and $g$ allow to consider a broad class of planar systems.
We will provide a list of specific applications in Section \ref{section-6}.
For the reader's convenience, we observe that, for the results to come,
we have in mind the model given by the scalar second order equation
\begin{equation}\label{eq-1.2}
x''+a_{\lambda,\mu}(t)(e^x-1)=0,
\end{equation}
which can be equivalently written as system \eqref{eq-1.1} with $h(y) = y$ and $g(x) = - 1 +\exp x$ (see Section \ref{section-3}).

We denote by $\Phi$ the Poincar\'{e} map associated with system \eqref{eq-1.1}. Recall that
$$\Phi(z) = \Phi_{0}^{T}(z):= (x(T;0,z),y(T;0,z)),$$
where $(x(\cdot\,;s,z),y(\cdot\,;s,z))$ is the solution of \eqref{eq-1.1} satisfying the initial condition $z=(x(s),y(s)).$
Since system \eqref{eq-1.1} has a Hamiltonian structure of the form
\begin{equation}\label{eq-H}
\begin{cases}
x'= \dfrac{\partial {\mathbf H}}{\partial y}(t,x,y)\\
y'= - \dfrac{\partial {\mathbf H}}{\partial x}(t,x,y)
\end{cases}
\end{equation}
for
${\mathbf H}(t,x,y) = a_{\lambda,\mu}(t)\mathcal{G}(x) + \mathcal{H}(y),$
the associated Poincar\'{e} map
is an area-pre\-serv\-ing homeomorphism, defined on a open set
$\Omega:=\text{dom}\Phi\subseteq {\mathbb R}^2,$
with $(0,0)\in \Omega.$ Thus a possible method to prove
the existence (and multiplicity) of $T$-periodic solutions can be based on the Poincar\'{e}-Birkhoff ``twist'' fixed point theorem.
A typical way to apply this result is to find a suitable annulus around the origin with radii $0 < r_0 < R_0$ such that for some $\mathfrak{a}
< \mathfrak{b}$
the twist condition
\begin{equation*}
\begin{cases}
\text{rot}_{z}(T) > \mathfrak{b},\quad\forall\, z \; \text{with }\, ||z||= r_0 \\
\text{rot}_{z}(T) < \mathfrak{a},\quad\forall\, z \; \text{with }\, ||z||= R_0
\end{cases}
\leqno{(TC)}
\end{equation*}
holds, where $\text{rot}_{z}(T)$ is the rotation number on the interval $[0,T]$
associated with the initial point $z\in {\mathbb R}^2 \setminus\{(0,0)\}.$ We recall that a possible definition of $\text{rot}_{z}$ for
equation \eqref{eq-1.1} is given by the integral formula
\begin{equation}\label{eq-1.3}
\text{rot}_{z}(t_1,t_2):= \frac{1}{2\pi} \int_{t_1}^{t_2} \frac{y(t) h(y(t)) + a_{\lambda,\mu}(t)x(t)g(x(t))}{x^2(t) + y^2(t)} dt,
\end{equation}
where $(x(t),y(t))$ is the solution of \eqref{eq-1.1} with
$(x(t_1),y(t_1)) = z\not= (0,0).$
For simplicity in the notation, we set
$$\text{rot}_{z}(T):= \text{rot}_{z}(0,T).$$
Notice that, due to the assumptions $h(s)s > 0$ and $g(s)s > 0$ for $s\not=0,$ it is convenient to use a formula like \eqref{eq-1.3}
in which the angular displacement is positive when the rotations around the origin are performed in the clockwise sense.

Under these assumptions, the Poincar\'{e}-Birkhoff theorem, in the version of \cite[Corollary 2]{Re-97},
guarantees that for each integer $j\in [\mathfrak{a},\mathfrak{b}],$
there exist \textit{at least two} $T$-periodic solutions of
system \eqref{eq-1.1}, having $j$ as associated rotation number.
In virtue of the first condition in $(C_{0}),$
it turns out that these solutions have precisely $2j$ simple transversal crossings with the $y$-axis
in the interval $[0,T[$ (see, for instance, \cite[Theorem 5.1]{BZ-13}, \cite[Theorem A]{MRZ-02}). Equivalently, for such a periodic solution
$(x(t),y(t))$, we have that $x$ has precisely $2j$ simple zeros
in the interval $[0,T[\,.$

If we look for $mT$-periodic solutions, we just consider the $m$-th iterate of the Poincar\'{e} map
$$\Phi^{(m)} = \Phi_{0}^{mT}$$
and assume the twist condition
\begin{equation*}
\begin{cases}
\text{rot}_{z}(mT) > \mathfrak{b},\quad\forall\, z \; \text{with }\, ||z||= r_0 \\
\text{rot}_{z}(mT) < \mathfrak{a},\quad\forall\, z \; \text{with }\, ||z||= R_0\,.
\end{cases}
\leqno{(TC_{m})}
\end{equation*}
From this point of view, the Poincar\'{e}-Birkhoff is a powerful tool to prove the existence of subharmonic solutions having $mT$ as minimal period.
Indeed, if we have $(TC_{m})$ satisfied with $\mathfrak{a} \leq j \leq \mathfrak{b},$ then we find $mT$-periodic solutions $(x(t),y(t))$ with
$x$ having exactly $2j$ simple zeros in $[0,mT[\,.$
In addition, if $j$ and $m$ are relatively prime integers, then it follows that these solutions cannot be $\ell T$-periodic for some
$\ell = 1,\dots, m-1.$ In particular, if $(x,y)$ is one of these $mT$-periodic solutions and $j=1$ or $j\geq 2$ is relatively prime with $m$ and
$T$ is the minimal period of the weight function $a_{\lambda,\mu}(t),$
then $mT$ will be the minimal period of $(x,y).$

Clearly, in order to apply this approach,  we need to have the Poincar\'{e} map defined on the closed disc
$D[R_{0}]$ of center the origin and radius $R_{0}$,
that is $D[R_{0}] \subseteq \Omega.$ Unfortunately, in general, the (forward)
global existence of solutions for the initial value problems is not guaranteed.
A classical counterexample can be found in \cite{CU-67} for
the superlinear equation $x'' + q(t) x^{2n+1} = 0$ (with $n\geq 1$),
where, even for a positive weight $q(t),$ the global existence of the solutions may fail. A typical feature of this class of counterexamples
is that solutions presenting a blow-up at some time $\beta^-,$ will make infinitely many winds around the origin as $t\to \beta^-.$
It is possible to overcome these difficulties by prescribing the rotation number for large solutions and using some truncation argument
on the nonlinearity, as shown in \cite{Ha-77},\cite{FS-16}.
In our case, the boundedness assumption at infinity given by one among $(g^{\pm}),((h^{\pm}),$ prevents such highly oscillatory phenomenon and guarantees
the continuability on $[0,T_1].$
The situation is even more complicated in the time intervals
where the weight function is negative \cite{BG-71}, \cite{Bu-76}.
Unless we impose some growth restrictions on the vector field in \eqref{eq-1.1}, for example that it has at most a linear growth at infinity,
in general we cannot prevent blow-up phenomena. Another possibility to avoid blow-up is to assume that both $(h_-)$ and $(h_+)$ hold
(or, alternatively, both $(g_-)$ and $(g_+)$). In fact, in this case, we have $x'$ bounded and hence $x(t)$
bounded in any compact time-interval and thus, from the second equation in \eqref{eq-1.1}  $y'$ is bounded on compact intervals, too.

\smallskip

With these premises, the following result holds.
\begin{theorem}\label{th-A}
Let $g,h: {\mathbb R} \to {\mathbb R}$ be locally Lipschitz continuous functions satisfying $(C_{0})$ and at least one between the four conditions
$(h_{\pm})$ and
$(g_{\pm}).$ Assume, moreover, the global continuability for the solutions of \eqref{eq-1.1}. Then, for each positive
integer $k,$ there exists $\Lambda_k > 0$
such that for each $\lambda > \Lambda_{k}$ and $j = 1,\dots, k, $ the system \eqref{eq-1.1}
has at least two $T$-periodic solutions $(x(t),y(t))$ with $x(t)$ having exactly $2j$-zeros in the interval $[0,T[\,.$
\end{theorem}
\noindent
Notice that in the above result we do not require any condition on the parameter $\mu > 0.$ On the other hand, we have to assume the global
continuability of the solutions, which in general is not guaranteed. Alternatively, we can exploit the fact that the solutions are globally defined
in the interval $[0,T_1]$ where $a\succ 0$ and, instead of asking for the global continuability on $[0,T],$ assume that $\mu$ is small enough.
Quite the opposite, for the next result we do not require the Poincar\'{e} map to be defined on the whole plane,
although now the parameter $\mu$ plays a crucial role and must be large instead.

\begin{theorem}\label{th-B}
Let $g,h: {\mathbb R} \to {\mathbb R}$ be locally Lipschitz continuous functions satisfying $(C_{0})$
and at least one between the four conditions $(h_{\pm})$ and
$(g_{\pm}).$ Then, for each positive
integer $k,$ there exists $\Lambda_k > 0$
such that for each $\lambda > \Lambda_{k}$ there exists $\mu^*= \mu^*(\lambda)$ such that for each $\mu > \mu^*$
and $j = 1,\dots, k, $ the system \eqref{eq-1.1}
has at least four $T$-periodic solutions $(x(t),y(t))$ with $x(t)$ having exactly $2j$-zeros in the interval $[0,T[\,.$
\end{theorem}
\noindent
The proofs of Theorem \ref{th-A} and Theorem \ref{th-B} are given in Section \ref{section-4}. In the case of Theorem \ref{th-B}
we will also show how the $2j$ zeros of $x(t)$ are distributed between the intervals $]0,T_1[$ and $]T_1,T[\,.$
Furthermore, we detect the presence of ``complex dynamics'',
in the sense that we can prove the existence of four compact invariant sets where the Poincar\'{e} map $\Phi$ is semi-conjugate
to the Bernoulli shift automorphism on $\ell =\ell_{\mu}\geq 2$ symbols.
A particular feature of Theorem \ref{th-B} lies in the fact that such result is robust with respect to small perturbations. In particular, it applies to
a perturbed Hamiltonian system of the form
\begin{equation}\label{eq-He}
\begin{cases}
x'= \dfrac{\partial {\mathbf H}}{\partial y}(t,x,y) + F_1(t,x,y,\varepsilon)\\
y'= - \dfrac{\partial {\mathbf H}}{\partial x}(t,x,y) +  F_2(t,x,y,\varepsilon)
\end{cases}
\end{equation}
with $F_1, F_2 \to 0$ as $\varepsilon \to 0$, uniformly in $t$, and for $(x,y)$ on compact sets.
Observe that
system \eqref{eq-He} has not necessarily a Hamiltonian structure and therefore it is no more guaranteed that the associated Poincar\'{e} map
is area-preserving.

In Section \ref{section-5}, versions of Theorem \ref{th-A} and Theorem \ref{th-B} for subharmonic solutions are given.
Moreover, in the setting of Theorem \ref{th-B}, we show that any $m$-periodic sequence on $\ell$ symbols can
be obtained by a $m$-periodic point of $\Phi$ in each of such compact invariant sets (see Theorem \ref{th-Bsub}).
Figure \ref{fig01} gives evidence of an abundance of subharmonic solutions to a system in the class \eqref{eq-1.1}.

\begin{figure}[htb]
\centering
\includegraphics[width=0.7\textwidth]{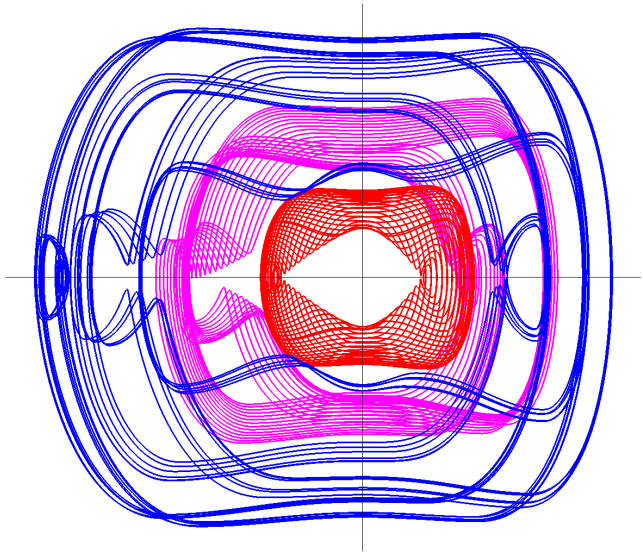}
\caption{A numerical simulation for system \eqref{eq-1.1}. The example is obtained for
$h(y) = y,$ $g(x)= - 1 + \exp x$ and
the weight function $a_{\lambda,\mu}$ where $a(t) = \sin(2\pi t)$,  $\lambda = 10$ and $\mu = 2.$
The phase-plane portrait is shown for the initial points $(0.4,0),$ $(0.1,0.2)$ and $(0.5,0).$}
\label{fig01}
\end{figure}

\section{The stepwise weight: a geometric framework}\label{section-3}

We focus on the particular case in which $h(y)=y$ and $g:{\mathbb R}\to {\mathbb R}$ is a locally Lipschitz continuous function such that
$$g(0) = 0 < g_{0}, \; \; g(x)x > 0\; \;\forall\, x\not=0, \;\text{with }\, g \; \text{ bounded on } \; {\mathbb R}^-,$$
so that $(C_{0})$ is satisfied along with $(g_{-}).$ A possible choice could be $g(x)= e^x-1$, but we stress that we do not ask for $g$ to be unbounded on
${\mathbb R}^+$.

The second order ordinary differential equation originating from \eqref{eq-1.1} reads
\begin{equation}\label{eq-2.1}
x''+a_{\lambda,\mu}(t)g(x)=0.
\end{equation}
In order to illustrate quantitatively the main ideas of the proof we choose a stepwise $T$-periodic function $a(\cdot)$ which takes value $a(t) = 1$
on an interval of length $T_1$ and value $a(t) = -1$ on a subsequent interval of length $T_2= T-T_1$, so that
$a_{\lambda,\mu}$ is defined as
\begin{equation}\label{eq-2.1b}
a_{\lambda,\mu}(t)=\left\{\begin{array}{ll}\lambda & \text{for } \;t\in[0,T_1[\\-\mu & \text{for } \; t\in [T_1,T_1+T_2[\end{array}\right.
\quad T_1 + T_2 = T.
\end{equation}
With this particular choice of $a(t),$ the planar system associated with \eqref{eq-2.1} turns out to be a periodic switched system \cite{Ba-14}.
Such kind of systems are widely studied in control theory.

For our analysis we first take into account the interval of positivity, where \eqref{eq-1.1} becomes
\begin{equation}\label{eq-2.2}
\begin{cases}
x'=y\\
y'=-\lambda g(x).
\end{cases}
\end{equation}
For this system the origin is a local center, which is global if
$\mathcal{G}(x)\to +\infty$ as $x\to \pm \infty$, where $\mathcal{G}(x)$ is the primitive of $g(x)$ such that $\mathcal{G}(0)=0$. The associated energy function is given by
\begin{displaymath}
E_1(x,y):=\dfrac{1}{2}y^2 + \lambda \mathcal{G}(x).
\end{displaymath}
For any constant $c$
with
$0 < c < \min\{\mathcal{G}(-\infty),\mathcal{G}(+\infty)\},$
the level line of \eqref{eq-2.2} of positive energy $\lambda c$ is a closed orbit $\Gamma$ which intersects the $x$-axis in the phase-plane
at two points $(x_-,0)$ and $(x_+,0)$ such that
$$x_- < 0 < x_+, \quad\text{and } \; c:= \mathcal{G}(x_-) = \mathcal{G}(x_+) > 0.$$
We call $\tau(c)$ the period of $\Gamma$, which is given by
$$\tau(c) = \tau^+(c) + \tau^-(c),$$
where
\begin{displaymath}
\tau^+(c):=\sqrt{\frac{2}{\lambda}}\int_{0}^{x_+}\dfrac{d\xi}{\sqrt{(c - \mathcal{G}(\xi))}}, \quad
\tau^-(c):=\sqrt{\frac{2}{\lambda}}\int_{x_-}^{0}\dfrac{d\xi}{\sqrt{(c - \mathcal{G}(\xi))}}
\end{displaymath}
The maps $c\mapsto \tau^{\pm}(c)$ are continuous. To proceed with our discussion, we suppose that $\mathcal{G}(-\infty) \leq \mathcal{G}(+\infty)$
(the other situation can be treated symmetrically). Then $\tau^-(c) \to +\infty$ as $c\to \mathcal{G}(-\infty)$
(this follows from the fact that $g(x)/x$ goes to zero as $x\to -\infty$, see \cite{Op-61}).
We can couple this result with an estimate near the origin
$$\displaystyle{\limsup_{c\to0^+}\tau(c)\leq{2\pi}/{\sqrt{\lambda g_{0}}}}$$
which follows from classical and elementary arguments.
\begin{proposition}\label{pr-2.1}
For each $\lambda > 0,$ the time-mapping $\tau$ associated with system \eqref{eq-2.2} is continuous and its range includes the interval
$]{2\pi}/{\sqrt{\lambda g_{0}}},+\infty[\,.$
\end{proposition}
Showing the monotonicity of the whole time-map $\tau(c)$ is, in general, a difficult task. However, for the exponential case
$g(x) = e^x -1$ this has been proved
in \cite{CW-86} (see also \cite{Ch-87}).

On the interval of negativity of $a_{\lambda,\mu}(t),$ system \eqref{eq-1.1} becomes
\begin{equation}\label{eq-2.3}
\begin{cases}
x'=y\\
y'=\mu g(x),
\end{cases}
\end{equation}
with $g(x)$ as above. For this system the origin is a global saddle with unbounded stable and  unstable manifolds contained in the zero level set of the
energy
\begin{displaymath}
E_2(x,y):=\dfrac{1}{2}y^2-\mu \mathcal{G}(x).
\end{displaymath}
In the following, given a point $P\in {\mathbb R}^2,$ we denote by $\gamma^{\pm}(P)$ and $\gamma(P)$ respectively
the positive/negative semiorbit and the orbit for the point $P$
with respect to the (local) dynamical system associated with \eqref{eq-2.3}.

If we start from a point $(0,y_0)$ with $y_0 > 0$ we can explicitly evaluate the blow-up time as follows. First of all we compute the
time needed to reach the level $x=\kappa > 0$ along the trajectory of \eqref{eq-2.3},
which is the curve of fixed energy $E_2(x,y) = E_2(0,y_0)$ with $y>0$.
Equivalently, we have
\begin{displaymath}
y=x'=\sqrt{y_0^2+2\mu \mathcal{G}(x)}
\end{displaymath}
from which
\begin{displaymath}
t=\int_0^{\kappa}\dfrac{dx}{\sqrt{y_0^2+2\mu \mathcal{G}(x)}}
\end{displaymath}
follows. Therefore, the blow-up time is given by
\begin{displaymath}
T(y_0)=\int_0^{+\infty}\dfrac{dx}{\sqrt{y_0^2+2\mu \mathcal{G}(x)}}\,.
\end{displaymath}
Standard theory guarantees that if the Keller - Osserman condition
\begin{equation}\label{intG}
\int_{}^{+\infty}\dfrac{dx}{\sqrt{\mathcal{G}(x)}}<+\infty
\end{equation}
holds, then the blow-up time is always finite and $T(y_0)\searrow0$ for $y_0\nearrow+\infty$.
On the other hand, $T(y_0)\nearrow +\infty$ for $y_0\searrow 0^+$.
Hence there exists $\bar{y}>0$ such that $T(y_0)>T_2$ for $y_0\in \,]0,\bar{y}[$ and hence there is no blow-up in $[T_1,T]$.

If we start with null derivative, i.e. from a point $(x_0,0)$, then similar calculations return
\begin{displaymath}
t=\int_{x_0}^{\kappa}\dfrac{dx}{\sqrt{2\mu(\mathcal{G}(x)-\mathcal{G}(x_0))}}
\end{displaymath}
and, since $\mathcal{G}(x)-\mathcal{G}(x_0)\sim g(x_0)(x-x_0)$ for $|x-x_0|\ll1$, the improper integral at $x_0$ is finite.
Therefore, the blow-up time is given by
\begin{displaymath}
T(x_0)=\int_{x_0}^{+\infty}\dfrac{dx}{\sqrt{2\mu(\mathcal{G}(x)-\mathcal{G}(x_0))}}\,.
\end{displaymath}
If \eqref{intG} is satisfied, then the blow-up time is always finite.
Moreover, $T(x_0)\to +\infty$ as $x_0\to 0^+.$
A similar but more refined result can be found in \cite[Lemma 3]{PaZa-00}.

Now we describe how to obtain Theorem \ref{th-A} and Theorem \ref{th-B} for system
\begin{equation}\label{eq-2.5}
\begin{cases}
x'=y\\
y'=-a_{\lambda,\mu}(t)g(x)
\end{cases}
\end{equation}
in the special case of a $T$-periodic stepwise function as in \eqref{eq-2.1b}. As we already observed,
equation \eqref{eq-2.5} is a periodic switched system and therefore its associated Poincar\'{e} map $\Phi$ on the interval $[0,T]$
splits as
$$\Phi = \Phi_2\circ\Phi_1\,$$
where $\Phi_1$ is the Poincar\'{e} map on the interval $[0,T_1]$ associated with system \eqref{eq-2.2} and
$\Phi_2$ is the Poincar\'{e} map on the interval $[0,T_2]$ associated with system \eqref{eq-2.3}.

\noindent
$(\mathbf{I})\;$ \textit{Proof of Theorem \ref{th-A} for the stepwise weight.}
We start by selecting a closed orbit $\Gamma^0$ near the origin of \eqref{eq-2.2} at a level energy $\lambda c_0$ and fix $\lambda$ sufficiently large,
say $\lambda > \Lambda_{k}$, so that in view of Proposition \ref{pr-2.1}
\begin{equation}\label{eq-2.7}
\tau(c_0) < \frac{T_1}{k+1}.
\end{equation}
Next, for the given (fixed) $\lambda$, we consider a second energy level $\lambda c_1$ with  $c_1 > c_0$ such that
\begin{equation}\label{eq-2.8}
\tau^-(c_1) > 2 T_2
\end{equation}
and denote by $\Gamma^1$ the corresponding closed orbit.
Let also
$${\mathcal A}:= \{(x,y): 2\lambda c_0 \leq y^2 + 2\lambda \mathcal{G}(x) \leq 2\lambda c_1\}$$
be the planar annular region enclosed between $\Gamma^0$ and $\Gamma^1.$
If we assume that the Poincar\'{e} map $\Phi_2$
is defined on ${\mathcal A},$ then the complete Poincar\'{e} map $\Phi$ associated with system \eqref{eq-2.5}
is a well defined area-preserving homeomorphism of the annulus ${\mathcal A}$ onto its image $\Phi({\mathcal A}) = \Phi_2({\mathcal A})$.
In fact the annulus is invariant under the action of $\Phi_1$.

During the time interval $[0,T_1],$ each point $z\in \Gamma^0$ performs $\lfloor T_1/\tau(c_0)\rfloor$
complete turns around the origin in the clockwise sense. This implies that
$$\text{rot}_z(0,T_1) \geq \left\lfloor \frac{T_1}{\tau(c_0)}\right\rfloor, \quad\forall\, z\in \Gamma^0\,.$$
On the other hand, from \cite[Lemma 3.1]{BZ-13} we know that
$$\text{rot}_z(T_1,T) =  \text{rot}_z(0,T_2)  > - \frac{1}{2}\,, \quad\forall\, z\not= (0,0)\,.$$
We conclude that
$$\text{rot}_z(T) > k, \quad\forall\, z\in \Gamma^0\,.$$
During the time interval $[0,T_1],$ each point $z\in \Gamma^1$ is unable to complete a full revolution around the origin,
because the time needed to cross either the second or the third quadrant is larger than $T_1.$ Using this information in connection to
the fact that the first and the third quadrants are positively invariant for the flow associated with \eqref{eq-2.3}, we find that
$$\text{rot}_z(T) < 1, \quad\forall\, z\in \Gamma^1\,.$$

An application of the Poincar\'{e}-Birkhoff fixed point theorem
guarantees for each $j= 1,\dots, k$ the existence of at least two
fixed points  $u_j= (u^j_x,u^j_y),$ $v_j=
(v^j_x,v^j_y)$ of the Poincar\'{e} map, with  $u_j, v_j$ in the interior of ${\mathcal A}$ and such that
$\text{rot}_{u_j}(T) = \text{rot}_{v_j}(T)= j.$ This in turns
implies the existence of at least two $T$-periodic solutions of
equation \eqref{eq-2.1} with $x(\cdot)$ having exactly $2j$-zeros
in the interval $[0,T].$
\finis

In this manner, we have proved Theorem \ref{th-A} for system \eqref{eq-2.5} in the special case of a stepwise weight function $a_{\lambda,\mu}$
as in \eqref{eq-2.1b}. Notice that no assumption on $\mu > 0$ is required. On the other hand, we have to suppose that $\Phi_2$ is
globally defined on ${\mathcal A}$.

\begin{remark}\label{rem-2.1}
From \eqref{eq-2.7} and the formulas for the period $\tau$ it is clear that assuming $T_1$ fixed and $\lambda$ large
is equivalent to suppose $\lambda$ fixed and $T_1$ large. This also follows from general considerations
concerning the fact that equation $x'' + \lambda g(x) = 0$ is equivalent to $u'' + \varepsilon^2 \lambda g(u) = 0$
for $u(\xi):= x(\varepsilon \xi).$
\finis
\end{remark}

\noindent
$(\mathbf{II})\;$ \textit{An intermediate step.}  We show how to improve the previous result if we add the condition that $\mu$ is sufficiently large.
First of all, we take $\Gamma^0$ and $\Gamma^1$ as before and $\lambda > \Lambda_k$ in order to produce the desired twist
for $\Phi$ at the boundary of ${\mathcal A}.$
Then we observe that the derivative of the energy $E_1$ along the trajectories of system \eqref{eq-2.3} is given by $(\lambda+\mu) y g(x),$
so it increases on the first and the third quadrant and decreases on the second and the fourth. Hence, if $\mu$ is sufficiently large,
we can find four arcs $\varphi_i\subseteq {\mathcal A}$, each one in the open $i$-th quadrant, with $\varphi_i$ joining $\Gamma^0$ and $\Gamma^1$
such that $\Phi_2(\varphi_i)$ is outside the region bounded by $\Gamma^1$ for $i=1,3$ and
$\Phi_2(\varphi_i)$ is inside the region bounded by $\Gamma^0$ for $i=2,4$. The corresponding position of ${\mathcal A}$ and $\Phi_2({\mathcal A})$
is illustrated in Figure \ref{fig03}.

\begin{figure}[htbp]
\centering
\includegraphics[width=0.7\textwidth]{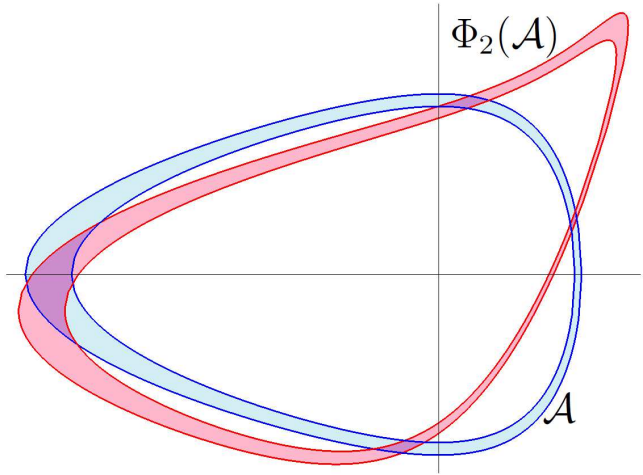}
\caption{A possible configuration of ${\mathcal A}$ and $\Phi_2({\mathcal A}).$ The example is obtained for $g(x) = - 1 + \exp x,$
$\lambda = \mu = 0.1$ and $T_2 =1.$ The inner and outer boundary $\Gamma^0$ and $\Gamma^1$ of the annulus ${\mathcal A}$ are the
energy level lines $E_1(x,y) = E_1(2,0)$ and $E_1(x,y) = E_1(2.1,0).$ To produce this geometry, the value of $T_1$ is not relevant because the
annulus is invariant for system \eqref{eq-2.2}. Since $\tau(c_0) < \tau(c_1),$ to have a desired twist condition,
we need to assume $T_1$ large enough.
}
\label{fig03}
\end{figure}

At this point, we enter in the setting of bend-twist maps. The arcs $\Phi_1^{-1}(\varphi_i)$
divide ${\mathcal A}$ into four regions, homeomorphic to rectangles.
The boundary of each of these regions can be split into two opposite sides contained in $\Gamma^0$ and $\Gamma^1$ and two other opposite sides
made by $\Phi_1^{-1}(\varphi_i)$ and $\Phi_1^{-1}(\varphi_{i+1})$ (mod$\,4$). On $\Gamma^0$ and $\Gamma^1$ we have the previously proved
twist condition on the rotation numbers, while on the other two sides we have $E_1(\Phi(P)) > E_1(P)$ for $P\in \Phi_1^{-1}(\varphi_i)$
with $i=1,3$ and $E_1(\Phi(P)) < E_1(P)$ for $P\in \Phi_1^{-1}(\varphi_i)$ with $i=2,4.$ Thus, using the Poincar\'{e}-Miranda theorem,
we obtain the existence of at least one fixed point
of the Poincar\'{e} map $\Phi$ in the interior of each of these regions.
In this manner, under an additional hypothesis of the form $\mu > \mu^*(\lambda),$ we improve Theorem \ref{th-A}
(for system \eqref{eq-2.5} and again in the special case of a stepwise weight),
finding at least four solutions with a given rotation number $j$ for $j=1,\dots, k.$
On the other hand, we still suppose that $\Phi_2$ is
globally defined on ${\mathcal A}$. The version of the bend-twist map theorem that we apply here is robust for small perturbations
of the Poincar\'{e} map, therefore the result holds also for some non-Hamiltonian systems whose vector field is close to that of \eqref{eq-2.1}.
\finis

\noindent
$(\mathbf{III})\;$  \textit{Proof of Theorem \ref{th-B} for the stepwise weight.}
First of all, we start with the same construction as
in part $(\mathbf{I})$ and choose $\Gamma^0,$ $\lambda > \Lambda_k$ according to \eqref{eq-2.7} and $\Gamma^1$ so that \eqref{eq-2.8} is satisfied.
Consistently with the previously introduced notation, we take
$$x_-^1 < x_-^0 < 0 < x_+^0 < x_+^1, \quad\text{with } \, \mathcal{G}(x_-^i) = c_i = \mathcal{G}(x_+^i), \; i = 0,1.$$
Notice that the closed curves $\Gamma^i$ intersect the coordinate axes at the points $(x_{\pm}^i,0)$ and $(0,\pm\sqrt{2\lambda c_i}).$
Next we choose $x_{\pm}^{\mu}$ and $y_0$ with
$$x_-^0 < x_-^{\mu} < 0 < x_+^{\mu} < x_+^0\,,\quad\text{and } \,  0 < y_0 < \sqrt{2\lambda c_0}$$
and define the orbits
$${\mathcal X}_{\pm}:=\gamma(x_{\pm}^{\mu},0), \quad  {\mathcal Y}_{\pm}:=\gamma(0,\pm y_0).$$
Setting
$$
{\mathcal T}({\mathcal X}_{\pm}):=\pm 2\int_{x_{\pm}^{\mu}}^{x_{\pm}^1}\dfrac{dx}{\sqrt{2\mu(\mathcal{G}(x)-\mathcal{G}(x_{\pm}^{\mu}))}}\,,\quad
{\mathcal T}({\mathcal Y}):= \int_{x_{-}^1}^{x_{+}^1}\dfrac{dx}{\sqrt{y_0^2+2\mu \mathcal{G}(x)}}
$$
we tune the values $x_{\pm}^{\mu},$ $y_0$ and $\mu$ so that
$$\max\{{\mathcal T}({\mathcal X}_{\pm}), {\mathcal T}({\mathcal Y})\} < T_2\,.$$
Clearly, given the other parameters, we can always choose $\mu$ sufficiently large, say $\mu > \mu^*$, so that the above condition is satisfied.

Finally, we introduce the stable and unstable manifolds, $W^{s}$ and $W^{u}$, for the origin as saddle point of system \eqref{eq-2.3}.
More precisely, we define the sets
$$W^{u}_{+}:= \{(x,y): E_2(x,y) = 0, x>0, y> 0\},$$
$$W^{s}_{-}:= \{(x,y): E_2(x,y) = 0, x<0, y> 0\},$$
$$W^{u}_{-}:= \{(x,y): E_2(x,y) = 0, x<0, y< 0\},$$
$$W^{s}_{+}:= \{(x,y): E_2(x,y) = 0, x>0, y< 0\},$$
so that $W^{s}=W^{s}_{-} \cup W^{s}_{+}$ and $W^{u}=W^{u}_{-} \cup W^{u}_{+}$.
The resulting configuration is illustrated in Figure \ref{fig04}.

\begin{figure}[htbp]
\centering
\includegraphics[width=0.7\textwidth]{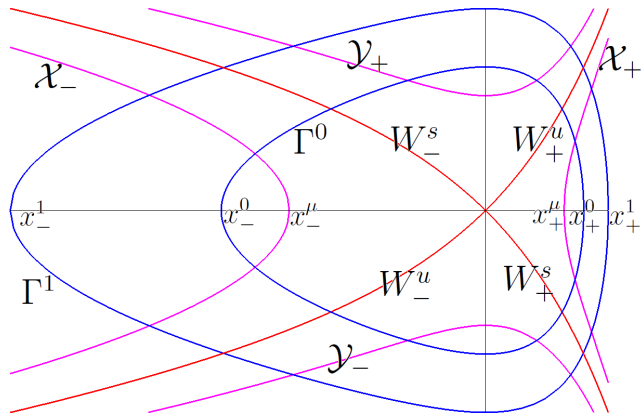}
\caption{The present figure shows the appropriate overlapping of the phase-portraits of systems \eqref{eq-2.2} and \eqref{eq-2.3}.
}
\label{fig04}
\end{figure}

The closed trajectories $\Gamma^{0},$ $\Gamma^{1}$ together with
${\mathcal X}_{\pm}\,,$ ${\mathcal Y}_{\pm}\,,$ $W^s_{\pm}$ and $W^u_{\pm}$ determine eight regions that we denote by ${\mathcal A}_{i}$
and ${\mathcal B}_{i}$ for $i=1,\dots,4,$ as in Figure \ref{fig05}.

\begin{figure}[htbp]
\centering
\includegraphics[width=0.7\textwidth]{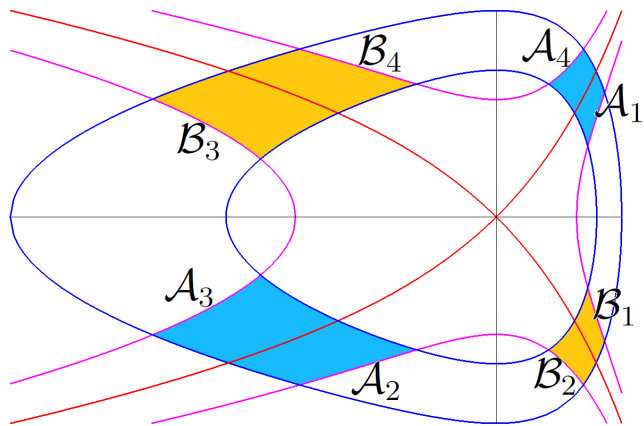}
\caption{The present figure shows the regions ${\mathcal A}_{i}$ and ${\mathcal B}_{i}$. We have labelled the regions following a clockwise order,
which is useful from the point of view of the dynamics.}
\label{fig05}
\end{figure}

Each of the regions  ${\mathcal A}_{i}$ and ${\mathcal B}_{i}$ is homeomorphic to the unit square and thus is a topological rectangle.
In this setting, we give an orientation to ${\mathcal A}_{i}$
by choosing ${\mathcal A}_{i}^-:= {\mathcal A}_{i}\cap (\Gamma^{0} \cup \Gamma^{1}).$ We take as
${\mathcal B}_{i}^-$ the closure of $\partial{\mathcal B}_{i}\setminus (\Gamma^{0} \cup \Gamma^{1}).$

We can now apply a result in the framework of the theory of topological horseshoes as presented in \cite{PaPiZa-08} and \cite{MRZ-10}.
Indeed, by the previous choice of $\lambda > \Lambda_k$ we obtain that
$$\Phi_1:\widehat{{\mathcal A}_{i}}\sap^k \widehat{{\mathcal B}_{i}}\,,\quad\forall\, i=1,\dots,4,$$
On the other hand, from $\mu > \mu^*$ it follows that
$$\Phi_2:\widehat{{\mathcal B}_{i}}\sap \widehat{{\mathcal A}_{i}}\,,\quad\forall\, i=1,\dots,4.$$
Then \cite[Theorem 3.1]{PaPiZa-08} (see also \cite[Theorem 2.1]{MRZ-10})
ensures the existence of at least $k$ fixed points for $\Phi = \Phi_2\circ\Phi_1$ in each of the regions ${\mathcal A}_{i}.$
This, in turns, implies the existence of $4k$ $T$-periodic solutions for system \eqref{eq-2.5}.

Such solutions are topologically different and can be classified, as follows: for each $j=1,\dots, k$ there is a solution
$(x,y)$ with
\begin{itemize}
\item[$\circ\;$] $(x(0),y(0))\in {\mathcal A}_1$ with $x(t)$ having $2j$ zeros in $]0,T_1[$ and strictly positive in $[T_1,T];$
\item[$\circ\;$] $(x(0),y(0))\in {\mathcal A}_2$ with $x(t)$ having $2j-1$ zeros in $]0,T_1[$  and one zero in $]T_1,T[;$
\item[$\circ\;$] $(x(0),y(0))\in {\mathcal A}_3$ with $x(t)$ having $2j$ zeros in $]0,T_1[$  and strictly negative in $[T_1,T];$
\item[$\circ\;$] $(x(0),y(0))\in {\mathcal A}_4$ with $x(t)$ having $2j-1$ zeros in $]0,T_1[$  and one zero in $]T_1,T[.$
\end{itemize}
In conclusion, for each $j=1,\dots,k$ we find at least four $T$-periodic solutions having precisely $2j$-zeros in $[0,T[.$ \finis

\begin{remark}\label{rem-2.2}
Having assumed that $g$ is bounded on ${\mathbb R}^-,$ we can also prove the existence of a $T$-periodic solution with $(x(0),y(0))\in {\mathcal A}_3$
and such that $x(t) < 0$ for all $t\in [0,T]$ while $y(t) = x'(t)$ has two zeros in $[0,T[.$ Moreover, the results from
\cite{MRZ-10,PaPiZa-08} guarantee also that each of the regions ${\mathcal A}_{i}$ contains a compact invariant set where
$\Phi$ is chaotic in the sense of Block and Coppel (see \cite{AK-01}).
\finis
\end{remark}

We further observe that, for equation \eqref{eq-2.1} the same results hold if condition $(g_-)$ is relaxed to
\begin{equation}\label{eq-2.9}
\lim_{x\to -\infty} \frac{g(x)}{x} = 0.
\end{equation}
In this manner we get the same number of four $T$-periodic solutions as obtained in \cite{BZ-13}. However, we stress that, even if the
conditions at infinity here and in that work are the same, nevertheless, the assumptions at the origin are completely different.
Indeed, in \cite{BZ-13} a one-sided superlinear condition in zero, of the form $g'(0^+) = 0$ or $g'(0^-)=0$ was required.
As a consequence, for $\lambda$ large, one could prove the existence of four $T$-periodic solutions with prescribed nodal properties
which come in pair, namely two ``small'' and two ``large''. In our case, if in place of $g_{0} > 0$ we assume $g'(0^+) = 0$ or $g'(0^-)=0,$
with the same approach we could prove the existence of eight $T$-periodic solutions, four ``small'' and four ``large''.

\bigskip

We conclude this section by observing that if we want to produce the same results for system \eqref{eq-1.1}, then we cannot replace
$(h_{\pm})$ or $(g_{\pm})$ with a weaker condition of the form of \eqref{eq-2.9}. Indeed, a crucial step in our proof is to have a twist condition,
that is a gap in the period between a fast orbit (like $\Gamma^0$) and slow one (like $\Gamma^1$). This is no more guaranteed for an autonomous system
of the form
\begin{equation*}
\begin{cases}
x'=h(y)\\
y'=-g(x)
\end{cases}
\end{equation*}
if $g(x)$ satisfies a sublinear condition at infinity as \eqref{eq-2.9}. Indeed, the slow decay of $g$ at infinity
could be compensated by a fast growth of $h$ at infinity. In \cite{CGM-00} the Authors
provide examples of isochronous centers for planar Hamiltonian systems even in the case when one of the two components is sublinear
at infinity.

\section{The general case: proof of the main results}\label{section-4}

Throughout the section, and consistently with Section \ref{section-2},
for each $s\in {\mathbb R}$ and $z\in {\mathbb R}^2,$ we denote by $(x(\cdot,s,z),y(\cdot;s,z))$ the solution of \eqref{eq-1.1} satisfying
$(x(s),y(s)) = z$ and, for $t\geq s,$ we set
$$\Phi_{s}^{t}(z):= (x(t,s,z),y(t;s,z)),$$
if the solution is defined on $[s,t].$

We prove both Theorem \ref{th-A}  (Theorem \ref{th-A1}) and Theorem \ref{th-B}  (Theorem \ref{th-B1}) assuming $(g_{-}).$ The proofs can be easily modified in order to take into account all the
other cases, namely $(g_{+})$, $(h_{-})$ or $(h_{+}).$ Concerning the $T$-periodic weight function we suppose
for simplicity that $a:{\mathbb R}\to {\mathbb R}$ is continuous and satisfies $(a^*).$
More general regularity conditions on $a(\cdot)$ can be considered as well.

\begin{theorem}\label{th-A1} Let $g, h$ be locally Lipschitz continuous functions.
Assume $(C_{0})$, $(g_{-})$ and the global continuability of the solutions. For each positive integer
$k$ there exists $\Lambda_k > 0$, such that for every $\lambda > \Lambda_k$ and $j=1,\dots,k,$ there are at least two $T$-periodic solutions for system
\eqref{eq-1.1} with $x$ having exactly $2j$-zeros in the interval $[0,T[\,.$
\end{theorem}
\begin{proof}
We split the proof into some steps in order to reuse some of them
for the proof of Theorem \ref{th-B1}.

\medskip

\noindent
\textit{Step 1. Evaluating the rotation number along the interval $[0,T_1]$ for small solutions.}
\\
Let $\varepsilon > 0$ be sufficiently small such that $g_0 - \varepsilon > 0$ and $h_0 -\varepsilon > 0$ and take $r^{\varepsilon} > 0$ such that,
by virtue of $(C_0),$
$$h(\xi)\xi \geq (h_0 -\varepsilon) \xi^2, \quad g(\xi)\xi \geq (g_0 -\varepsilon) \xi^2, \quad\forall\, |\xi|\leq r^{\varepsilon}.$$
Let $u(t):=(x(t),y(t))$ be a solution of \eqref{eq-1.1} such that $0 < ||u(t)||\leq r^{\varepsilon}$ for all $t\in [0,T_1].$
We consider the modified clockwise rotation number associated with the solution $u(\cdot)$ in the interval $[0,T_1]$
(which is the interval where $a \succ 0$),
defined as
$$\text{Rot}^p(u;0,T_1):= \frac{\sqrt{p}}{2\pi}\int_{0}^{T_1}\frac{h(y(t))y(t) + \lambda a^+(t)g(x(t))x(t)}{py^2(t) + x^2(t)}\,dt,$$
where $p > 0$ is a fixed number that will be specified later.
The modified rotation number can be traced back to the classical Pr\"{u}fer transformation
and it was successfully applied in \cite{FH-93} (see also \cite{Za-96} and the references therein).
A systematic use of the modified rotation number in the context of the Poincar\'{e}-Birkhoff theorem, with all the needed technical details, is
exhaustively described by Boscaggin in \cite{Bo-11}. Here we follow the same approach.
The key property of the number $\text{Rot}^p(u)$ is that, when for some $p$  it assumes an integer value, that same value
is independent on the choice of $p.$ Moreover, as a consequence of $h(y)y > 0$ for $y\not=0,$ we have that
if $\tau_1 < \tau_2$ are two consecutive zeros of $x(\cdot)$, then
$\text{Rot}^p(u;\tau_1,\tau_2) = 1/2$ (independently on $p$).
Hence, we can choose suitably the constant $p$ in order to estimate in a simpler way the rotation number.
In our case, if we take
\begin{equation*}
p:= \frac{1}{\lambda}\frac{h_0 - \varepsilon}{(g_0 - \varepsilon)|a^+|_{L^{\infty}(0,T_1)}}\,,
\end{equation*}
and recall that we are evaluating the rotation number on a ``small'' solution $u$ so that $h(y(t))y(t) \geq (h_0 - \varepsilon)y^2(t)$
and $g(x(t))x(t) \geq (g_0 - \varepsilon)x^2(t),$
we find
\begin{align*}
&\text{Rot}^p(u;0,T_1) \geq \frac{\sqrt{p}}{2\pi}\int_{0}^{T_1}
\frac{(h_0 - \varepsilon)y^2(t) + \lambda a^+(t)(g_0 - \varepsilon)x^2(t)}{py^2(t) + x^2(t)}\,dt\\
&\phantom{\qquad}= \frac{\sqrt{p}}{2\pi}\int_{0}^{T_1}
\frac{(\lambda(g_0 - \varepsilon)|a^+|_{L^{\infty}(0,T_1)})p y^2(t) + \lambda a^+(t)(g_0 - \varepsilon)x^2(t)}{py^2(t) + x^2(t)}\,dt\\
&\phantom{\qquad}\geq \frac{\lambda (g_0 - \varepsilon)\sqrt{p}}{2\pi} \int_{0}^{T_1} a^+(t) \, dt = \frac{\sqrt{\lambda}}{2\pi} \kappa(\varepsilon),
\end{align*}
where
$$\kappa(\varepsilon): = \left( \frac{(h_0 - \varepsilon)(g_0 - \varepsilon)}{|a^+|_{L^{\infty}(0,T_1)}} \right)^{1/2}\int_{0}^{T_1}  a^+(t) \, dt > 0.$$
Hence, given any positive integer $k,$ we can take
\begin{equation}\label{Lambdak}
\Lambda_{k}:= \left(\frac{2\pi}{\kappa(\varepsilon)}\right)^2(k+1)^2
\end{equation}
so that, for each $\lambda > \Lambda_{k}$ we obtain that
$\text{Rot}^p(u;0,T_1) > k+1$ and therefore, by \cite[Proposition 2.2]{Bo-11}, $\text{rot}_{z}(0,T_1) > k+1,$ for $z= u(0)$ and
$\text{rot}_{z}$ defined in \eqref{eq-1.3}.

\medskip

\noindent
\textit{Step 2. Evaluating the rotation number along the interval $[0,T]$ for small solutions.}
\\
Consider now the interval $[T_1,T]$ where $a(t) \leq 0.$ In this case, we are in the same situation as in \cite[Lemma 3.1]{BZ-13}
and the corresponding result implies that $\text{rot}_{z_1}(T_1,T) > -1/2,$
for $z= u(T_1)$. As a consequence, we conclude that if
$u(t):=(x(t),y(t))$ is a solution of \eqref{eq-1.1} such that $0 < ||u(t)||\leq r^{\varepsilon}$ for all $t\in [0,T_1]$ and $\lambda > \Lambda_{k}$,
then $\text{rot}_{z}(T):=\text{rot}_{z}(0,T) > k$ for $z= u(0).$

\medskip

\noindent
\textit{Step 3. Consequences of the global continuability.}
\\
The global continuability of the solutions implies
the fulfillment of the so-called ``elastic property'' (cf. \cite{Kr-68},\cite{CMZ-90}). In our case, recalling also that
nontrivial solutions never hit the origin, we obtain
\begin{itemize}
\item[$(i_1)\;$] for each $r_1 > 0$ there exists $r_2 \in ]0,r_1[$ such that $||z|| \leq r_2$
implies $0 < ||\Phi_{0}^{t}(z)|| \leq r_1\,,\;\forall\, t\in [0,T];$
\item[$(i_2)\;$] for each $R_1 > 0$ there exists $R_2 > R_1$ such that $||z|| \geq R_2$
implies $||\Phi_{0}^{t}(z)|| \geq R_1\,,\;\forall\, t\in [0,T]$
\end{itemize}
(see \cite[Lemma 2]{Za-96}).

\medskip

\noindent
\textit{Step 4. Rotation numbers for small initial points.}
\\
Suppose now that $\varepsilon > 0$ and $\lambda > \Lambda_k$ are chosen as in Step 1 and let $\mu > 0$ be fixed.
Using $(i_1)$ in Step 3 we determine a small radius
$r_0 = r_0(\varepsilon,\lambda,\mu)> 0$
such that for each initial point $z\in {\mathbb R}^2$ with $||z|| = r_0$ it follows that
the solution $u(t)=(x(t),y(t))$ of \eqref{eq-1.1} with $u(0) = z$ satisfies $0 < ||u(t)||\leq r^{\varepsilon}$ for all $t\in [0,T_1].$
Hence, by Step 2 we conclude that
\begin{equation}\label{smallrotations}
\text{rot}_{z}(T) > k, \quad\forall\, z \; \text{with }\, ||z||= r_0\,.
\end{equation}

\medskip

\noindent
\textit{Step 5. Evaluating the rotation number along the interval $[0,T]$ for large solutions.}
\\
Suppose, from now on, that $\lambda$ and $\mu$ are fixed as in Step 4.
Let $u(t)=(x(t),y(t))$ be any nontrivial solution of \eqref{eq-1.1}
which crosses the third quadrant in the phase-plane.
If this happens, we can assume that there is an interval $[\alpha,\beta] \subseteq [0,T]$ such that
$x(t) \leq 0$ for all $t\in [\alpha,\beta]$ with $x(\alpha) = 0 = y(\beta)$ and $x(t) < 0$ for all $t\in\,]\alpha,\beta]$
as well as $y(t) < 0$ for all $t\in [\alpha,\beta[$. Note that, from the first equation in
\eqref{eq-1.1}, when $y(t) = 0,$ also $x'(t) = 0.$
Assumption
$(g_{-})$ implies there is a bound, say $M$ for $g(x)$ when $x\leq 0,$ namely $|g(x)| \leq M$ for all $x\leq 0.$
Thus, integrating the second equation
in system \eqref{eq-1.1} we get that, for all $t\in [\alpha,\beta],$ the following estimate holds:
\begin{align*}
|y(t)| &= \left|y(\beta) + \int_{\beta}^{t} a_{\lambda,\mu}(s)g(x(s))\,ds\right|\\
&\leq M \int_{\alpha}^{\beta} a_{\lambda,\mu}(s)\,ds \leq M \left(\lambda \int_{0}^{T_1} a^+(t)\,dt +  \mu \int_{T_1}^{T} a^-(t)\,dt \right)=: M_1\,.
\end{align*}
Now, integrating the first equation of the same system, we obtain the estimate below, for all $t\in [\alpha,\beta]$:
$$|x(t)| = \left|x(\alpha) + \int_{\alpha}^{t} h(y(s))\,ds\right| \leq M_2:= (\beta - \alpha) \max\{|h(y)|\,;\, |y|\leq M_1\}.$$
With a similar argument, it is easy to check that the same bounds for $|y(t)|$ and $|x(t)|$ hold if the solution crosses the second quadrant
instead of the third one. Hence,
any solution that in a time-interval $[\alpha,\beta]$ crosses the third quadrant, or the second quadrant,
is such that $|y(t)|\leq M_1$ and $|x(t)|\leq M_2$ for all $t\in [\alpha,\beta].$

In view of the above estimates and arguing by contradiction, we can then conclude that if the solution $u$ satisfies
$$||u(t)||\geq M_3:=1 + \left(M_1 + M_2\right)^{1/2},\quad \forall\, t\in [0,T],$$
then for $u(\cdot)$ is impossible to cross the third quadrant and it is also impossible to cross the second one.

\medskip

\noindent
\textit{Step 6. Rotation numbers for large initial points.}
\\
Using $(i_2)$ in Step 3 we determine a large radius
$R_0 = R_{0}(\lambda,\mu) > 0$
such that for each initial point $z\in {\mathbb R}^2$ with $||z|| = R_0$ it follows that
the solution $u(t)=(x(t),y(t))$ of \eqref{eq-1.1} with $u(0) = z$ satisfies $||u(t)||\geq M_3$ for all $t\in [0,T].$
Hence, by Step 5 we conclude that
\begin{equation}\label{largerotations}
\text{rot}_{z}(T) < 1, \quad\forall\, z \; \text{with }\, ||z||= R_0\,.
\end{equation}
Indeed, if, by contradiction, $\text{rot}_{z}(T) \geq 1,$ then the solution $u(t)$ of \eqref{eq-1.1} with $u(0) = z$ must cross at least once
one between the third and the second quadrant and this fact is forbidden by the choice of $R_0$ which implies that $||u(t)||\geq M_3$
for all $t\in [0,T].$

\medskip

\noindent
\textit{Step 7. Applying Poincar\'{e}-Birkhoff fixed point theorem.}
\\
At this point we can to conclude the proof. From \eqref{smallrotations} and \eqref{largerotations}
we have the twist condition $(TC)$ satisfied for $\mathfrak{b}=k$ and $\mathfrak{a}=1$ and the thesis follows as explained in the introductory discussion preceding the
statement of Theorem \ref{th-A}.
\qed

\end{proof}

\medskip
\begin{remark}\label{rem-3.1}{In view of the above proof, a few observations are in order.

1. We think that the choice of $\Lambda_k$ in \eqref{Lambdak}, although reasonably good,
is not the optimal one. One could slightly improve it, by using some comparison argument with the rotation numbers
associated with the limiting linear equation
$$x' = h_0 y, \quad y' = -\lambda g_0 a^+(t) x.$$
We do not discuss further this topic in order to avoid too much technical details in the proof.

2. In the statement of Theorem \ref{th-A1} we have explicitly recalled the assumptions on $g$ and $h$ to be locally Lipschitz continuous functions
so to have a well defined (single-valued) Poincar\'{e} map. With this respect, we should mention
that there is a recent version of the Poincar\'{e}-Birkhoff theorem due to Fonda and Ure\~{n}a \cite{FU-16}, \cite{FU-17} which, for Hamiltonian systems
like \eqref{eq-1.1}, does not require the uniqueness of the solutions for the initial value problems and just the continuability of the solutions
on $[0,T]$ is needed. The theorems in \cite{FU-17} apply to higher dimensional Hamiltonian systems as well. For another recent application
of such resuts to planar systems, in which the uniqueness of the solutions of the Cauchy problems is not required,
see also \cite{COZ-16}. In our case, even if we apply the Fonda-Ure\~{n}a theorem, we still need to assume at least an upper bound on $g(x)/x$
and $h(y)/y$ near zero, so to avoid the possibility that a (nontrivial) solution
$u(\cdot)$ of \eqref{eq-1.1} with $u(0)\not= (0,0)$ may hit the origin at some time $t\in \,]0,T],$
thus preventing the rotation number to be well defined.
See \cite[Section 4]{Bu-78} for a detailed discussion of these aspects.}
\finis
\end{remark}

Now we are in position to give the proof of Theorem \ref{th-B} (which is presented as Theorem \ref{th-B1} below). For the next result we
do not assume the global continuability of the solutions. Accordingly, both $h$ (at $\pm\infty$) and $g$ (at $+\infty)$ may have a superlinear growth.

For the foregoing proof we recall that we denote by $D(R)$ and $D[R]$ the open and the closed disc in ${\mathbb R}^2$
of center the origin and radius $R > 0.$ Given $0 < r < R,$ we denote by $A[r,R]$ the closed annulus $A[r,R]:= D[R]\setminus D(r).$ Let also $Q_i$
for $i=1,2,3,4$ be the usual quadrants of ${\mathbb R}^2$ counted in the natural counterclockwise sense starting from
$$Q_1:=\{(x,y): x\geq 0, y\geq 0\}.$$

\begin{theorem}\label{th-B1} Let $g, h$ be locally Lipschitz continuous functions satisfying $(C_{0})$.
Suppose also that $(g_{-})$ holds. For each positive integer
$k$ there exists $\Lambda_k > 0$, such that for every $\lambda > \Lambda_k$ there exists $\mu^* = \mu^*(\lambda)$
such that for each $\mu > \mu^*$ and $j=1,\dots,k,$ there are at least four $T$-periodic solutions for system
\eqref{eq-1.1} with $x$ having exactly $2j$-zeros in the interval $[0,T[\,.$
\end{theorem}
\begin{proof}
For our proof, we will take advantage of some steps already settled in the proof of Theorem \ref{th-A1}.

First of all, we consider system \eqref{eq-1.1} on the interval $[0,T_1]$, so that the system can be written as
\begin{equation}\label{eq-reduced}
x' = h(y), \qquad y' = - \lambda a^+(t) g(x)
\end{equation}
and observe that all the solutions of \eqref{eq-reduced} are globally defined on $[0,T_1].$ To prove this fact, we
observe that the sign assumptions on $h$ and $g$ in $(C_0)$ and $(g_-)$ guarantee that \eqref{eq-reduced} belongs to the class of equations for which
the global continuability of the solutions was proved in \cite{DiZa-96}. Hence our claim is proved.

Now, we repeat the same computations as in the Steps 1-3-4-5-6 of the preceding proof (with only minor modifications, since now we work on $[0,T_1]$
instead of $[0,T]$) and, having fixed $\lambda >\Lambda_{k}$ (with the same constants $\Lambda_k$ as in \eqref{Lambdak}),
we are in the following setting:
{\em
\begin{enumerate}
\item[$(TC_*)$\;] There are constants $r_0 = r_{0}(\varepsilon,\lambda)$ and $R_0 = R_{0}(\lambda)$, with $0 < r_{0} < R_{0}$, such that\\
$\displaystyle{\text{rot}_{z}(T_1) > k+1,\,\forall\, z: \, ||z||= r_0\,; \;\;
\text{rot}_{z}(T_1) < 1, \,\forall\, z: \, ||z||= R_0\,.}$
\end{enumerate}
}

By a classical compactness argument following by the global continuability of the solutions of \eqref{eq-reduced} we can determine
two positive constants $\mathfrak{s}_0=\mathfrak{s}_0(\lambda,r_0,R_0)$ and $\mathfrak{S}_0=\mathfrak{S}_0(\lambda,r_0,R_0)$ with
$$0 < \mathfrak{s}_0 < r_0 < R_0 < \mathfrak{S}_0\,,$$
such that
\begin{equation}\label{eq-frak}
\mathfrak{s}_0 \leq ||\Phi_0^t(z)|| \leq \mathfrak{S}_0\,,\;\; \forall\, t\in [0,T_1]\,,\;\; \forall\, z: \; r_0\leq ||z||\leq R_{0}\,.
\end{equation}
We introduce now the following sets, which are all annular sectors and hence topological rectangles according to the terminology
of the Introduction.
$${\mathcal P}_1:= A[r_0,R_0]\cap Q_1\,,\quad {\mathcal P}_2:= A[r_0,R_0]\cap Q_3\,,$$
$${\mathcal M}_1:= A[\mathfrak{s}_0,\mathfrak{S}_0]\cap Q_4\,,\quad {\mathcal M}_2:= A[\mathfrak{s}_0,\mathfrak{S}_0]\cap Q_2\,.$$
To each of these sets we give an \textit{orientation}, by selecting a set $[\cdot]^-$ which is the union of two disjoint arcs of its boundary, as follows.
$${\mathcal P}_i^-:= {\mathcal P}_i\cap \partial A[r_0,R_0], \quad \widehat{{\mathcal P}_{i}}:=({\mathcal P}_i,{\mathcal P}_i^-), \;\; i=1,2,$$
$${\mathcal M}_i^-:= {\mathcal M}_i\cap \{(x,y): xy = 0\}, \quad \widehat{{\mathcal M}_{i}}:=({\mathcal M}_i,{\mathcal M}_i^-), \;\; i=1,2.$$

To conclude the proof, we show that for each integer $j=1,\dots, k$ and $i=1,2$
there is a pair of compact disjoint sets $H'_{i,j}\,, H''_{i,j}
\subseteq {\mathcal P}_i$ such that
\begin{equation}\label{sap}
(H,\Phi): \widehat{{\mathcal P}_{i}}\sap \widehat{{\mathcal P}_{i}}\,\quad i=1,2,
\end{equation}
where $H$ stands for  $H'_{i,j}$ or  $H''_{i,j}$ and $\Phi:= \Phi_{0}^{T}.$ Along the proof we will also check that the $4k$ sets
$H'_{i,j}$ and $H''_{i,j}$ for $i=1,2$ and $j=1,\dots,k$ are pairwise disjoint. A fixed point theorem introduced in \cite{PaZa-02}
and recalled in Section \ref{section-1}
(see \cite[Theorem 3.9]{PaZa-04b} for the precise formulation which is needed in the present situation) ensures the existence of at least a
fixed point for the Poincar\'{e} map $\Phi$ in each of the $4k$ sets $H'_{i,j}$ and $H''_{i,j}$.

To prove \eqref{sap} we proceed with two steps. First we show that, for any fixed $j\in \{1,\dots k\}$ there is a
compact set $H'_{1,j} \subseteq {\mathcal P}_{1}$ such that
\begin{equation}\label{sap11}
(H'_{1,j},\Phi_{0}^{T_1}): \widehat{{\mathcal P}_{1}}\sap \widehat{{\mathcal M}_{1}}
\end{equation}
and another compact set
$H''_{1,j} \subseteq {\mathcal P}_{1}$ such that
\begin{equation}\label{sap12}
(H''_{1,j},\Phi_{0}^{T_1}): \widehat{{\mathcal P}_{1}}\sap \widehat{{\mathcal M}_{2}}\,.
\end{equation}
In the same manner we also prove that there are disjoint compact sets $H'_{2,j}\,,$ $H''_{2,j}\subseteq {\mathcal P}_{2}$ such that
\begin{equation}\label{sap21}
(H'_{2,j},\Phi_{0}^{T_1}): \widehat{{\mathcal P}_{2}}\sap \widehat{{\mathcal M}_{2}}
\end{equation}
and
\begin{equation}\label{sap22}
(H''_{2,j},\Phi_{0}^{T_1}): \widehat{{\mathcal P}_{2}}\sap \widehat{{\mathcal M}_{1}}\,.
\end{equation}
Next, we prove that
\begin{equation}\label{sap2}
\Phi_{T_1}^{T}: \widehat{{\mathcal M}_{i}}\sap \widehat{{\mathcal P}_{\ell}}, \quad \forall\, i=1,2, \; \forall\, \ell = 1,2.
\end{equation}
Clearly, once all the above relations have been verified, we obtain \eqref{sap}, using the composition $\Phi = \Phi_{T_1}^{T}\circ \Phi_{0}^{T_1}$
and counting correctly all the possible combinations.

\medskip
\noindent
\textit{Proof of \eqref{sap11}.} We choose a system of polar coordinates $(\theta,\rho)$
starting at the positive $y$-axis and counting the positive rotations in the
clockwise sense, so that, for $z\not=0$ and $t\in [0,T_1],$ $\theta(t,z)$ denotes the angular coordinate associated with the solution
$u = (x,y)$ of \eqref{eq-reduced} with $u(0) = z.$ For $z\in {\mathcal P}_1$ we already know that $u(t) \in A[\mathfrak{s}_0,\mathfrak{S}_0]$
for all $t\in [0,T_1].$ For any fixed $j\in\{1,\dots,k\}$ we define
$$H'_{1,j}:= \{z\in {\mathcal P}_1\,:\, \theta(T_1,z)\in [(\pi/2) + 2j\pi,\pi + 2j\pi]\}.$$
Note that an initial point $z\in {\mathcal P}_1$ belongs to $H'_{1,j}$ if and only if $\Phi_{0}^{T_1}(z)\in {\mathcal M}_1$
with $x(\cdot)$ having precisely $2j$ zeros in the interval $]0,T_1[\,.$ Then it is clear that $H'_{1,j_1}\cap H'_{1,j_2}=\emptyset$
for $j_1\not=j_2\,.$

Let $\gamma:[a_0,a_1]\to {\mathcal P}_1$ be a continuous map such that $||\gamma(a_0)|| = r_0$ and $||\gamma(a_1)|| = R_0\,,$
that is $\gamma$ is a path contained in (with values in) ${\mathcal P}_1$ and meeting the opposite sides of ${\mathcal P}^-_1\,.$
For simplicity in the notation, we set
$$\vartheta(t,\xi):= \theta(t,\gamma(\xi)).$$
As we have previously observed $\Phi_{0}^t(\gamma(\xi))\in A[\mathfrak{s}_0,\mathfrak{S}_0]$ for all $t\in [0,T_1]$ and $\xi\in [a_0,a_1].$
From property $(TC_*)$ we have that
$$\vartheta(T_1,a_0) > \vartheta(0,a_0) + 2(k+1)\pi \geq 2(k+1)\pi \geq 2(j+1)\pi,$$
while
$$\vartheta(T_1,a_1) < \vartheta(0,a_0) + 2\pi \leq \frac{\pi}{2} + 2\pi.$$
The continuity of the map $\varphi:[a_0,a_1]\ni \xi\mapsto \vartheta(T_1,\xi)$ implies that the range of $\varphi$ covers the interval
$[(\pi/2) + 2j\pi,\pi + 2j\pi].$ Therefore there exist $b_0,b_1$ with
$a_0 < b_0 < b_1 < a_1$ such that $\vartheta(T_1,b_0) = \pi + 2j\pi,$ $\vartheta(T_1,b_1) = (\pi/2) + 2j\pi$ and
$$\vartheta(T_1,\xi) \in \left[\frac{\pi}{2} + 2j\pi,\pi + 2j\pi\right], \quad \forall\, \xi \in [b_0,b_1].$$
If we denote by $\sigma$ the restriction of the path $\gamma$ to the subinterval $[b_0,b_1],$ we have that $\sigma$ has values in $H'_{1,j}$
and, moreover the path $\Phi_{0}^{T_1}\circ \sigma$ has values in ${\mathcal M}_1$ and connects the two components of ${\mathcal M}^-_1\,.$
Thus the validity of \eqref{sap11} is checked.

\medskip
\noindent
\textit{Proof of \eqref{sap12}.} This is only a minor variant of the preceding proof and, with the same setting and notation as above, we just define
$$H''_{1,j}:= \{z\in {\mathcal P}_1\,:\, \theta(T_1,z)\in [(3\pi/2) + 2(j-1)\pi,2j\pi]\}.$$
An initial point $z\in {\mathcal P}_1$ belongs to $H''_{1,j}$ if and only if $\Phi_{0}^{T_1}(z)\in {\mathcal M}_2$
with $x(\cdot)$ having exactly $2j-1$ zeros in the interval $]0,T_1[\,.$ Then it is clear that $H''_{1,j_1}\cap H''_{1,j_2}=\emptyset$
for $j_1\not=j_2\,.$ Moreover, it is also evident that all the sets $H'$ and $H''$ are pairwise disjoint. The rest of the proof follows the same steps
as the previous one, with minor modifications and using again the crucial property $(TC_*).$

\medskip
\noindent
\textit{Proof of \eqref{sap21} and \eqref{sap22}.} Here we follow a completely symmetric argument by defining a family of pairwise disjoint
compact subsets of ${\mathcal P}_2$ as
$$H'_{2,j}:= \{z\in {\mathcal P}_2\,:\, \theta(T_1,z)\in [(3\pi/2) + 2j\pi,2(j+1)\pi]\}$$
and
$$H''_{2,j}:= \{z\in {\mathcal P}_2\,:\, \theta(T_1,z)\in [(\pi/2) + 2j\pi,\pi +2j\pi]\},$$
respectively.
Note that we consider an initial point $z\in {\mathcal P}_2$ with an associated angle $\theta(0,z)\in [\pi,(3/2)\pi].$
The rest of the proof is a mere repetition of the arguments presented above.

\medskip
\noindent
\textit{Proof of \eqref{sap2}.} We have four conditions to check, but it is clear that it will be sufficient to prove
\begin{equation}\label{sap2n}
\Phi_{T_1}^{T}: \widehat{{\mathcal M}_{1}}\sap \widehat{{\mathcal P}_{1}}\;\; \text{and } \;
\Phi_{T_1}^{T}: \widehat{{\mathcal M}_{1}}\sap \widehat{{\mathcal P}_{2}}\,,
\end{equation}
the other case, being symmetric. In this situation, we have only to repeat step by step the argument described in \cite[pp. 85-86]{PaZa-04a} where a similar situation is taken into account. There is however a
substantial difference between the equation considered in \cite{PaZa-04a} and our equation, that in the interval $[T_1,T]$
can be written as
\begin{equation}\label{eq-reducedbis}
x' = h(y), \qquad y' = \mu a^-(t) g(x).
\end{equation}
Indeed, in our case we cannot exclude that some solutions are not globally defined on $[T_1,T],$ due to the presence of possible blow-up phenomena
in some quadrants. To overcome this difficulty, we follow an usual truncation argument. More precisely, recalling the ``large''
constant $\mathfrak{S}_0$ introduced in \eqref{eq-frak}, we define the truncated functions
$$\hat{g}(x):=
\left\{
\begin{array}{lll}
g(-\mathfrak{S}_0), \; &\forall \, x\leq - \mathfrak{S}_0\\
g(x), \; &\forall\, x\in [-\mathfrak{S}_0,\mathfrak{S}_0]\\
g(\mathfrak{S}_0), \; &\forall \, x\geq \mathfrak{S}_0
\end{array}
\right.
$$
and
$$\hat{h}(y):=
\left\{
\begin{array}{lll}
h(-\mathfrak{S}_0) + y +\mathfrak{S}_0, \; &\forall \, y\leq - \mathfrak{S}_0\\
h(y), \; &\forall\, y\in [-\mathfrak{S}_0,\mathfrak{S}_0]\\
h(\mathfrak{S}_0) + y - \mathfrak{S}_0, \; &\forall \, y\geq \mathfrak{S}_0
\end{array}
\right.
$$
which are locally Lipschitz continuous with $g$ bounded and $h$ having a linear growth. Now the uniqueness and the global existence of
the solutions in the interval $[T_1,T]$ for the solutions of the truncated system
\begin{equation}\label{eq-truncated}
x' = \hat{h}(y), \qquad y' = \mu a^-(t) \hat{g}(x)
\end{equation}
is guaranteed.
Apart for few minor details we can closely follow the proof of \cite[Theorem 3.1]{PaZa-04a} for the interval when the weight function is negative.
Of course, now the result will be valid for the Poincar\'{e} map $\hat{\Phi}_{T_1}^{T}$ associated with system
\eqref{eq-truncated} for the time-interval $[T_1,T].$
If we treat with such technique the Poincar\'{e} map $\hat{\Phi}_{T_1}^{T}\,,$ we can prove that
any path $\gamma$ in ${\mathcal M}_1$ joining the two sides of ${\mathcal M}^-_1$ contains a sub-path $\sigma$ with
$\bar{\sigma}\subseteq {\mathcal M}^-_1$ such that $\hat{\Phi}_{T_1}^t(\bar{\sigma}) \subset D[0,\mathfrak{S}_0]$ for all
$t\in [T_1,T]$ and such that $\hat{\Phi}_{T_1}\circ \sigma$ is a path in ${\mathcal P}_1$ joining the opposite sides of ${\mathcal P}^-_{1}\,.$
This in turn implies
\begin{equation*}
\hat{\Phi}_{T_1}^{T}: \widehat{{\mathcal M}_{1}}\sap \widehat{{\mathcal P}_{1}}.
\end{equation*}
On the other hand, the condition $\hat{\Phi}_{T_1}^t(\bar{\sigma}) \subset D[0,\mathfrak{S}_0]$ for all
$t\in [T_1,T]$ implies that $\hat{\Phi}_{T_1}^t(\bar{\sigma}) = {\Phi}_{T_1}^t(\bar{\sigma})$ for all
$t\in [T_1,T]$ and therefore we have that
\begin{equation*}
{\Phi}_{T_1}^{T}: \widehat{{\mathcal M}_{1}}\sap \widehat{{\mathcal P}_{1}}.
\end{equation*}
All the other instances of \eqref{sap2} can be verified in the same manner. This completes the proof of the theorem.
\qed
\end{proof}

As a byproduct of the method of proof we have adopted, we are able to classify the nodal properties of the four $T$-periodic solutions as follows.
\begin{proposition}\label{pr-B1} Let $g, h$ be locally Lipschitz continuous functions satisfying $(C_{0})$
and at least one between the four conditions
$(h_{\pm})$ and $(g_{\pm})$.
For each positive integer
$k$ there exists $\Lambda_k > 0$, such that for every $\lambda > \Lambda_k$ there exists $\mu^* = \mu^*(\lambda)$
such that for each $\mu > \mu^*$ and $j=1,\dots,k,$ there are at least four $T$-periodic solutions for system
\eqref{eq-1.1} which can be classified as follows:
\begin{itemize}
\item[$\circ\;$] one solution with $x(0) > 0, x'(0) > 0$ and with $x(\cdot)$ having exactly $2j$ zeros in $]0,T_1[$ and no zeros in $[T_1,T];$
\item[$\circ\;$] one solution with $x(0) > 0, x'(0) > 0$ and with $x(\cdot)$ having exactly $2j-1$ zeros in $]0,T_1[$ and one zero in $[T_1,T];$
\item[$\circ\;$] one solution with $x(0) < 0, x'(0) < 0$ and with $x(\cdot)$ having exactly $2j$- zeros in $]0,T_1[$ and no zeros in $[T_1,T];$
\item[$\circ\;$] one solution with $x(0) < 0, x'(0) < 0$ and with $x(\cdot)$ having exactly $2j-1$ zeros in $]0,T_1[$ and one zero in $[T_1,T].$
\end{itemize}
\end{proposition}

\medskip
\begin{remark}\label{rem-3.2}{We present some observations related to the proof of Theorem \ref{th-B1}.

1. In view of the results in \cite{PaZa-04a} and the notation recalled in the Introduction,
our proof implies that
$$
\Phi_{0}^{T}: \widehat{{\mathcal P}_{i}}\sap^{2k} \widehat{{\mathcal P}_{i}}\,\quad i=1,2.
$$
As a consequence, there are two compact invariant sets contained in ${\mathcal P}_1$
and ${\mathcal P}_2\,,$ respectively, where $\Phi_{0}^{T}$ induces chaotic dynamics on $2k$ symbols. Also for any periodic
sequence of symbols, subharmonic solutions, associated with that periodic sequence, do exist (see \cite{MPZ-09}).

2. In principle, our approach could be extended to differential systems in which the weight function displays a finite number
of positive humps separated by negative ones. Although the feasibility of this study is quite clear, we have not pursued this line of research
in this article for sake of conciseness.

3. A comparison between Theorem \ref{th-A1} and Theorem \ref{th-B1} suggests that it could be interesting to present examples
of differential systems in which there are exactly two (respectively four) $T$-periodic solutions with given nodal properties
and then discuss the change in the number of solutions using $\mu$ as a bifurcation parameter.}
\finis
\end{remark}

\section{Subharmonic solutions}\label{section-5}

In this Section we briefly discuss how to adapt the proofs of Theorem \ref{th-A} and Theorem \ref{th-B} to obtain subharmonic solutions.
Throughout the Section we suppose that $a:{\mathbb R}\to {\mathbb R}$ is a continuous $T$-periodic weight function
satisfying $(a^*).$ As before, more general regularity conditions on $a(\cdot)$ can be considered.

Speaking of subharmonic solutions, we must observe that if $u=(x,y)$ is a $mT$-periodic solution of \eqref{eq-1.1},
then also $u_i(\cdot):= u(\cdot - iT)$ is a $mT$-periodic solution for all $i=1,\dots,m-1.$ Such solutions, although distinct, are
considered to belong to the same periodicity class.

First of all, we look at the proof of Theorem \ref{th-A1} and give estimates on the rotation numbers on the interval $[0,mT]$
for some integer $m\geq 2.$ Iterating the argument in Step 1-4 and taking a smaller $r_0$ if necessary, we can prove that for
$\lambda > \Lambda_{k}$ as in \eqref{Lambdak}, we obtain
\begin{equation}\label{smallrotationsm}
\text{rot}_{z}(mT) > mk, \quad\forall\, z \; \text{with }\, ||z||= r_0\,.
\end{equation}
Repeating the computation in Step 5-6 for the interval $[0,mT]$ and taking a larger $R_0$ if necessary, we get
\begin{equation}\label{largerotationsm}
\text{rot}_{z}(mT) < 1, \quad\forall\, z \; \text{with }\, ||z||= R_0\,.
\end{equation}
In this manner we have condition $(TC_{m})$ satisfied for $\mathfrak{b}=mk$ and $\mathfrak{a}=1.$ Now, if we fix an integer $j\in \{1,\dots,mk\}$
which is relatively prime with $m,$ we obtain at least two $mT$-periodic solutions of system \eqref{eq-1.1} with $x(\cdot)$ having exactly
$2j$ simple zeros in the interval $[0,mT[$. As $m$ and $j$ are coprime numbers, these solutions cannot be $\ell T$-periodic for some
$\ell\in \{1,\dots,m-1\}.$ Since, by $(a^*),$ $T$ is the minimal period of $a(\cdot),$ we conclude that $mT$ is the minimal period of the
solution $(x,y)$ of \eqref{eq-1.1} (see \cite{DiZa-93}, \cite{DiZa-96} for previous related results and how to prove
the minimality of the period via the information about the rotation number).
In this manner, the following result is proved.

\begin{theorem}\label{th-Asub}
Let $g,h: {\mathbb R} \to {\mathbb R}$ be locally Lipschitz continuous functions satisfying $(C_{0})$
and at least one between the four conditions
$(h_{\pm})$ and
$(g_{\pm}).$ Assume, moreover, the global continuability for the solutions of \eqref{eq-1.1}. Let $m\geq 2$ be a fixed integer.
Then, for each positive
integer $k,$ there exists $\Lambda_k > 0$
such that for each $\lambda > \Lambda_{k}$ and $j = 1,\dots, mk,$ with $j$ relatively prime with $m$, the system \eqref{eq-1.1}
has at least two periodic solutions $(x(t),y(t))$ of minimal period $mT$, not belonging to the same periodicity class.
\end{theorem}
\noindent
Also for Theorem \ref{th-Asub} the same observation as in Remark \ref{rem-3.1}.2 applies, namely, using Fonda-Ure\~{n}a version of the
Poincar\'{e}-Birkhoff theorem, we can remove the local Lipschitz condition outside the origin.

\medskip

Looking for an extension of Theorem \ref{th-B} to the case of subharmonic solutions, in view of
Remark \ref{rem-3.2}.1 and \cite{PaZa-04a},\cite{MPZ-09}, the condition
$$
\Phi_{0}^{T}: \widehat{{\mathcal P}_{i}}\sap^{2k} \widehat{{\mathcal P}_{i}}\,\quad i=1,2
$$
implies the following property with respect to periodic solutions. Let ${\mathcal P}= {\mathcal P}_i$ for $i=1,2.$
There exists $2k$ pairwise disjoint compact sets $S_1,\dots S_{2k} \subseteq {\mathcal P}$ such that for each
$m$-periodic two-sided sequence $\xi:= (\xi_n)_{n\in{\mathbb Z}},$ with $\xi_n\in \{1,\dots,2k\}$ for all $n\in {\mathbb Z},$
there exists a fixed point $z^*$ of $\Phi_{0}^{mT}$ such that $\Phi_{0}^{nT}(z^*)\in S_{\xi_n}\,, \forall\, n\in {\mathbb Z}.$
In this case we say that the trajectory associated with the initial point $z^*$ follows the periodic \textit{itinerary}
$(\dots,S_{\xi_0},\dots,S_{\xi_n},\dots).$

From this observation,
the following result holds.

\begin{theorem}\label{th-Bsub}
Let $g,h: {\mathbb R} \to {\mathbb R}$ be locally Lipschitz continuous functions satisfying $(C_{0})$
and at least one between the four conditions
$(h_{\pm})$ and
$(g_{\pm}).$ Let $m\geq 2$ be a fixed integer. Then, for each positive
integer $k,$ there exists $\Lambda_k > 0$
such that for each $\lambda > \Lambda_{k}$ there exists $\mu^*= \mu^*(\lambda)$ such that for each $\mu > \mu^*$ the following property holds:
given any periodic two-sided sequence $(\xi_n)_{n\in{\mathbb Z}},$ with $\xi_n\in \{1,\dots,2k\}$ for all $n\in {\mathbb Z},$
with minimal period $m$, the system \eqref{eq-1.1}
has at least two periodic solutions $(x(t),y(t))$ of minimal period $mT$, following an itinerary of sets
associated with $(\xi_n)_{n\in{\mathbb Z}}$ and not belonging to the same periodicity class.
\end{theorem}
\noindent
A simple comparison of the two theorems shows that when $m$ grows, the number of different subhamonics found by Theorem \ref{th-Bsub}
largely exceeds the number of those obtained by Theorem \ref{th-Asub}. The number of $m$-order subharmonics can be precisely
determined by a combinatorial formula coming from the study of \textit{aperiodic necklaces} \cite{Fe-18}.

\section{Examples and applications}\label{section-6}

We propose a few examples of equations, coming from the literature, which fit into the framework of our theorems.

As a first case, we consider the following Duffing type equation with relativistic acceleration
\begin{equation}\label{eq-6.1}
(\varphi(u'))'+ a_{\lambda,\mu}(t)g(u) = 0,
\end{equation}
where
\begin{displaymath}
\varphi(s)=\dfrac{s}{\sqrt{1-s^2}}.
\end{displaymath}
A variant of \eqref{eq-6.1} in the form of $(\varphi(u'))' + g(t,u) = 0$ is considered in \cite{BG-13} where pairs of
periodic solutions with prescribed nodal properties are found by means of the Poincar\'{e}-Birkhoff theorem
for a function $g(t,u)$ having at most linear growth in $u$ at infinity.
Equation \eqref{eq-6.1} can be equivalently written as
\begin{align*}
\begin{cases}
x'=\varphi^{-1}(y)\\
y'=-a_{\lambda,\mu}(t)g(x)
\end{cases}
\end{align*}
which is the same as system \eqref{eq-1.1} with $h=\varphi^{-1}.$ In this case, since $\varphi^{-1}$ is bounded, both $(h_{-})$ and $(h_{+})$
are satisfied and the global existence of the solutions is guaranteed.
Moreover all the other assumptions required for $h$ in $(C_{0})$ are satisfied with $h_0 =1.$ Hence all the results of the paper
apply to \eqref{eq-6.1} once we assume that $(C_0)$ holds for $g.$ Notice that we do not need any growth assumption on $g.$

\medskip

Our second example is inspired by the work of Le and Schmitt \cite{LS-95} where the authors proved the existence of $T$-periodic solutions
for the second order equation
\begin{equation}\label{eq-6.2}
u'' + k(t)e^u = p(t),
\end{equation}
with $k,p$ $T$-periodic functions, with $k$ changing sign, $p$ with zero mean value and such that
$$\int_{0}^{T} k(t)e^{u_0(t)}dt < 0,$$
where $u_0(t)$ denotes a $T$-periodic solution of $u'' = p(t).$ If
we call $\tilde{u}(t)$ the $T$-periodic solution of \eqref{eq-6.2} whose existence is guaranteed by \cite[Remark 6.4]{LS-95},
and set
\begin{displaymath}
u(t)=x(t)+\tilde{u}(t),
\end{displaymath}
then \eqref{eq-6.2} is transformed to the equivalent equation
\begin{displaymath}
x''+q(t)(e^x-1)=0
\end{displaymath}
with
\begin{displaymath}
q(t):=k(t)e^{\tilde{u}(t)}
\end{displaymath}
which changes sign if and only if $k(t)$ changes sign. Thus we enter in the setting of \eqref{eq-1.2} and Theorem \ref{th-B}
can be applied. Note that in general, in an interval where $k<0$ we may have blow-up of the solutions in the first quadrant of the phase-plane
and thus the Poincar\'{e} map cannot be defined on a whole (large) annulus surrounding the origin.
Clearly, in order to apply our theorem,
we will just need to adapt our conditions on the weight $a_{\lambda,\mu}(t)$ to the coefficients $k(t)$ and $p(t).$

\medskip

As a last example, we consider a model adapted from the classical Lotka-Volterra predator-prey system.
We take into account the system
\begin{align*}
\begin{cases}
P'=P(-c_1(t)+d(t)N)\\
N'=N(c_2(t)-b(t)P)\\
\end{cases}
\end{align*}
which represents the dynamics of a prey population $N(t)>0$ under the effect of a predator population $P(t)>0.$ Notice that the order in which the
two equations appear in the above system is not the usual one but it is convenient so to enter the setting of system \eqref{eq-1.1}.
All the coefficients
$b,c_1,c_2,d$ are continuous and $T$-periodic functions.
In \cite{DiZa-96} the existence of infinitely many subharmonic solutions was proved under the assumption that $c_1(t)$ and $c_2(t)$ have positive average
and $b(t)\succ 0,$ $d(t)\succ 0$ in $[0,T]\,.$ Extensions to higher dimensional systems have been recently obtained in \cite{FT-18} under similar sign conditions on the coefficients.
Some results about the stability of the solutions for this model are obtained as a special case in \cite{LOT-96}. For the main Lotka-Volterra model thereby proposed the search of subharmonics solutions has been recently addressed \cite{LoMu-18} under some specific assumptions on the averages of $b(t)$ and $d(t)$, in a framework that can be regarded as a counterpart to ours.

Our aim now is to discuss a case in which $b(t)$ may be negative on some subinterval. We perform a change of variables setting $u=\log P$ and $v=\log N$ and obtain the new system
\begin{align*}
\begin{cases}
u'=-c_1(t)+d(t)e^v\\
v'=c_2(t)-b(t)e^u.
\end{cases}
\end{align*}
Suppose now a $T$-periodic solution $(\tilde{u}(t),\tilde{v}(t))$ is given. With a further change of variables we can write the generic solutions as
\begin{align*}
u(t)=x(t)+\tilde{u}(t)\\
v(t)=y(t)+\tilde{v}(t)
\end{align*}
and by substitution in the previous system we arrive at
\begin{align*}
\begin{cases}
x'=d(t)e^{\tilde{v}(t)}(e^y-1)\\
y'=-b(t)e^{\tilde{u}(t)}(e^x-1)\,.
\end{cases}
\end{align*}
At this point, we can adapt the coefficients in order to enter in the frame of system \eqref{eq-1.1}. More precisely, we suppose that
$d(t)e^{\tilde{v}(t)} \equiv \text{constant} = D >0$, so that $h(y):= D(e^y-1),$ $g(x):= e^x-1$ and set
\begin{displaymath}
q(t):=b(t)e^{\tilde{u}(t)},
\end{displaymath}
which changes sign if and only if $b(t)$ changes sign. Thus we enter in the setting of \eqref{eq-1.2} and Theorem \ref{th-B}
can be applied. The same remark of the previous case holds. In particular, in order to apply our theorem,
we need to translate our conditions on the weight $a_{\lambda,\mu}(t)$ to the coefficients $b(t).$


\end{document}